\documentclass[10pt]{amsart}
\usepackage[utf8]{inputenc}
\usepackage{amsfonts}
\usepackage{graphics}
\usepackage[all, cmtip]{xy}
\usepackage{amsthm,amsfonts,amssymb,amsmath,amsxtra}
\usepackage[all]{xy}
\usepackage{amsmath}
\usepackage{flexisym}
\usepackage{mathrsfs}
\usepackage{amsfonts}
\usepackage{enumitem}
\usepackage{amsmath}
\usepackage[colorlinks]{hyperref}
\hypersetup{
  citecolor   = blue %Colour of citations
 }
\usepackage{amsmath}
\newcommand{\rmM}{{\rm M}}

\newcommand{\End}{{\rm End}}

\newcommand{\Res}{{\rm Res}}
\newcommand{\Gal}{{\rm Gal}}

\newcommand{\Gr}{{\rm Gr}}

\newcommand{\GL}{{\rm GL}}

\newcommand{\ZZ}{\mathbb{Z}}

\newcommand{\QQ}{\mathbb{Q}}
\newcommand{\GG}{\mathbb{G}}

\newcommand{\RR}{\mathbb{R}}
\newcommand{\CC}{\mathbb{C}}

\newcommand{\calF}{\mathcal{F}}
\newcommand{\calL}{\mathcal{L}}

\newcommand{\calR}{\mathcal{R}}

\newcommand{\scrG}{\mathscr{G}}

\newcommand{\bb }{\langle}
\newcommand{\pp}{\rangle}
\newcommand{\llp}{(\!(}
\newcommand{\rlp}{)\!)}
\newcommand{\lp}{[\![}
\newcommand{\rp}{]\!]}

\newcommand{\Mspl}{{\rm M}^{\rm spl}}

\usepackage{multicol}
\numberwithin{equation}{subsection}

\newtheorem{Theorem}{Theorem}[section]
\newtheorem{Remark}[Theorem]{Remark}
\newtheorem{Remarks}[Theorem]{Remarks}
\newtheorem{Lemma}[Theorem]{Lemma}
\newtheorem{Proposition}[Theorem]{Proposition}
\newtheorem{Corollary}[Theorem]{Corollary}

\newtheorem{Main Theorem}[Theorem]{Main Theorem}
\newtheorem{Definition}[Theorem]{Definition}
\usepackage[colorlinks]{hyperref}
\makeatletter
\renewcommand*\env@matrix[1][*\c@MaxMatrixCols c]{%
  \hskip -\arraycolsep
  \let\@ifnextchar\new@ifnextchar
  \array{#1}}
\makeatother

\newif\ifgrading

\gradingfalse

\usepackage[colorlinks]{hyperref}
\hypersetup{linkcolor=black}
\usepackage{xcolor}
\usepackage{wasysym}

\usepackage{tikz-cd}

\newcommand{\calG}{{\mathcal G}}
\newcommand{\calO}{{\mathcal O}}
\newcommand{\Mloc}{{\rm M}^{\rm loc}}
\newcommand{\Spec}{{\rm Spec \, } }
\newcommand{\wti}{\widetilde}

\usepackage{xcolor}

\newcommand{\Addresses}{{% additional braces for segregating \footnotesize
		\bigskip
		\footnotesize

\textsc{Centre de Mathématiques Laurent Schwartz (CMLS), CNRS, École polytechnique, Institut Polytechnique de Paris, 91120 Palaiseau, France}\par\nopagebreak
		\textit{E-mail address:} \texttt{stephane.bijakowski@polytechnique.edu}\\
  
		\textsc{Department of Mathematics, Universität Münster, Münster, 48149, Germany}\par\nopagebreak
		\textit{E-mail address:} \texttt{io.zachos@uni-muenster.de}\\
		
		\textsc{Morningside Center of Mathematics, Chinese Academy of Sciences,
			Beijing, 100190, China}\par\nopagebreak
		\textit{E-mail address:} \texttt{zhihaozhao@amss.ac.cn}
}}	

\begin{document}

\title{On the geometry of splitting models} 
	\date{}
	\author{S. Bijakowski, I. Zachos and Z. Zhao}
	\maketitle

	\begin{abstract}
We consider Shimura varieties associated to a unitary group of signature $(n-s,s)$ where $n$ is even. For these varieties, by using the spin splitting models from \cite{ZacZhao1}, we construct flat, Cohen-Macaulay, and normal $p$-adic integral models with reduced special fiber and with an explicit moduli-theoretic description over odd primes $p$ which ramify in the imaginary quadratic field with level subgroup at $p$ given by the stabilizer of a $\pi$-modular lattice in the hermitian space. We prove that the special fiber of the corresponding splitting model is stratified by an explicit poset with a combinatorial description, similar to \cite{BH}, and we describe its irreducible components. Additionally, we prove the closure relations for this stratification.
%We prove that the special fiber is stratified by an explicit poset with a combinatorial description, and in particular we have a description of the irreducible components of the special fiber. Moreover we prove the closure relations for this stratification.
	\end{abstract}
	
\tableofcontents
\section{Introduction}\label{Intro}
\subsection{}
The goal of this paper is to construct ``nice" $p$-adic integral models of Shimura varieties over places of bad reduction. We consider Shimura varieties of PEL type, i.e. moduli spaces
of abelian varieties with polarization, endomorphisms and level structure, which are associated to unitary similitude groups of signature $(r, s)$
over an imaginary quadratic field $K$. These Shimura varieties have canonical models over the “reflex” number field $E$ which is the field of rational numbers $\mathbb{Q}$ if $r = s$ and $E = K$ otherwise.

One of the basic problems regarding the arithmetic of Shimura varieties is
the construction of such well-behaved integral models over the ring of integers $O_E$. The behavior of these depends very much on the “level subgroup”. Here, the level subgroup is the stabilizer of a $\pi$-modular lattice in an even dimensional $n$ hermitian space. (This is a lattice which is equal to its dual times a uniformizer.) %For this level subgroup, and for any signature $(n-s, s)$, we construct $p$-adic integral models which are flat, normal, with a reduced special fiber and with an explicit moduli-theoretic description. 

By using work of Rapoport-Zink \cite{RZbook} we first construct $p$-adic integral models, which have simple and explicit moduli descriptions but are not always flat. These models are \'etale locally isomorphic to certain simpler schemes the \textit{naive local models}. By the work of Pappas-Rapoport \cite{PR2}, we consider a variation of the above moduli problem, the splitting models, where we add in the moduli problem an additional subspace in the deRham filtration $ {\rm Fil}^0 (A) \subset H_{dR}^1(A)$ of the universal abelian variety $A$, which satisfies certain conditions. In \cite{ZacZhao1}, the second and third author showed that these models are not flat for any signature $(r,s)$. Then by adding the ``spin condition" they considered the \textit{spin splitting models} and they proved that the models are smooth for the signature $(n-1,1)$ and have semi-stable reduction for $(n-2,2)$ and $(n-3,3)$. Here we study these models for higher signatures and, using methods employed in the work of the first author with V. Hernandez \cite{BH}, we analyze the special fiber of these models. %Following the work of the second and third author \cite{ZacZhao1}, we consider a variation of the above moduli problem, \textit{the {\color{blue}spin} splitting models}, where we add in the moduli problem an additional subspace in the deRham filtration $ {\rm Fil}^0 (A) \subset H_{dR}^1(A)$ of the universal abelian variety $A$, which satisfies certain conditions. This is essentially an instance of the notion of a ``linear modification" introduced in \cite{P}. 
%by employing similar methods that were used in the work of the first author with V. Hernadez \cite{BH}, we study the special fiber of these models. 
In particular, we introduce and study a similar combinatorial stratification for which we can compute the closure relations. Then for any signature $(n-s,s)$, we show that the spin splitting models, which have an explicit moduli-theoretic description, are normal, Cohen-Macaulay and flat with reduced special fiber. Lastly, as an application we compare how this stratification relates to the Kottwitz-Rapoport stratification. 

It turns out that studying splitting models is a very fruitful way to get nice integral models of Shimura varieties and even integral models with semi-stable reduction. This method has been exploited in many cases including the Weil restriction cases in \cite{PR2} and \cite{BHUnr} and the ramified unitary cases in \cite{Kr}, \cite{Richarz}, \cite{Zac1}, \cite{ZacZhao1}, \cite{ZacZhao2} and \cite{HLS}. It is worth mentioning that in many of these works, the authors frequently consider variants of splitting models to obtain an “honest” splitting model, highlighting the need for a more general, group-theoretic definition of splitting models. Lastly, we want to mention that splitting models have found applications to the study of arithmetic intersections of
special cycles and Kudla’s program. For example, the Krämer splitting model \cite{Kr}, which has semi-stable reduciton, was used in the study of the modularity of arithmetic special cycles \cite{BHKRY} and the proof of  the Kudla-Rapoport conjecture \cite{HLSY}.

\subsection{} Let us give some details. Recall that $K$ is an imaginary quadratic field and we fix an embedding $\varepsilon: K\rightarrow \mathbb{C}$. Let $W$ be a $n$-dimensional $K$-vector space, equipped with a non-degenerate hermitian form $\phi$. Consider the group $G = GU_n$ of unitary similitudes for $(W,\phi)$ of even dimension $n=2m > 3$ over $K$. We fix a conjugacy class of homomorphisms $h: \Res_{\mathbb{C}/\mathbb{R}}\mathbb{G}_{m,\mathbb{C}}\rightarrow GU_n$ corresponding to a Shimura datum $(G,X_h)$ of signature $(n-s,s)$. Set $X = X_h$ and assume that $s\leq n-s$. The pair $(G,X)$ gives rise to a Shimura variety $Sh(G,X)$ over the reflex field $E$. Let $p$ be an odd prime number and assume that it is ramified in $K$. Set $K_1=K\otimes_{\QQ} \QQ_p$, $V=W\otimes_{\QQ} \QQ_p$ and let $\pi$ be a uniformizer of $K_1$. %For any $O_{K_1}$-lattice $\Lambda$ in $V$ the uniformizing element $\pi$ induces a $O_{K_1}$- linear mapping on $\Lambda$ which we denote by $t$. 
Assume that the hermitian form $\phi$ splits on $V$, i.e. there is a basis $e_1, \dots, e_n$ such that $\phi(e_i, e_{n+1-j})=\delta_{ij}$ for $1\leq i, j\leq n$, and define 
\[
\Lambda_i = \text{span}_{O_{K_1}} \{\pi^{-1}e_1, \dots, \pi^{-1}e_i, e_{i+1}, \dots, e_n\}
\]
the standard lattices in $V$. % We can complete this into a self dual lattice chain by setting $\Lambda_{i+kn}:=\pi^{-k}\Lambda_i$ (see \S \ref{Prel.1}).
%By \cite{PR}, there are 3 different types of the special maximal parahoric subgroups of $GU_n$. %, i.e. parahoric subgroups corresponding to a vertex in the building which is special in the sense of Bruhat-Tits theory. 
Next, consider the stabilizer subgroup
\[
P_{m}:=\{g\in GU_n\mid g\Lambda_m=\Lambda_m \}
\]
which is a special maximal parahoric subgroup in the sense of Bruhat Tits theory \cite{Tits} (see also \cite{PR} for more details).

Choose a sufficiently small compact open subgroup $K^p$ of the prime-to-$p$ finite adelic points $G({\mathbb A}_{f}^p)$ of $G$ and set $\mathbf{K}=K^pP_{\{m\}}$. The Shimura variety  ${\rm Sh}_{\mathbf{K}}(G, X)$ with complex points
 \[
 {\rm Sh}_{\mathbf{K}}(G, X)(\mathbb{C})=G(\mathbb{Q})\backslash X\times G({\mathbb A}_{f})/\mathbf{K}
 \]
is of PEL type and has a canonical model over the reflex field $E$. We set $\mathcal{O} = O_{E_v}$ where $v$ the unique prime ideal of $E$ above $(p)$.

We consider the moduli functor $\mathcal{A}^{\rm naive}_{\mathbf{K}}$ over $\Spec \mathcal{O} $ given in \cite[Definition 6.9]{RZbook}:\\
A point of $\mathcal{A}^{\rm naive}_{\mathbf{K}}$ with values in the %$\Spec \mathcal{O} $
$\mathcal{O}  $-scheme $S$ is the isomorphism class of the following set of data $(A,\bar{\lambda},\iota, \bar{\eta})$:
\begin{enumerate}
    \item $A$ is an abelian scheme over $S$ of dimension $n$.
    \item $\lambda$ is a polarization, principal at $p$.
    \item $\iota: \mathcal{O} \rightarrow \text{End}(A)$, making the Rosati involution and the complex conjugation compatible.
    \item $ \bar{\eta} $ is a level structure away from $p$.
\end{enumerate}
(See \S \ref{Int} for more details.) % Here $\mathcal{G} = \underline{{\rm Aut}}(\mathcal{L})$ is the (smooth) group scheme over $\mathbb{Z}_p$ with $P_{\{m\}} = \mathcal{G}(\mathbb{Z}_p)$ the subgroup of $G(\mathbb{Q}_p)$ fixing the lattice chain $\mathcal{L}$.   
The functor $\mathcal{A}^{\rm naive}_{\mathbf{K}}$ is representable by a quasi-projective scheme over $\mathcal{O}$ and there is a natural isomorphism
\[
\mathcal{A}^{\rm naive}_{\mathbf{K}} \otimes_{\calO} E_v = {\rm Sh}_{\mathbf{K}}(G, X)\otimes_{E} E_v.
\]

As is explained in \cite{RZbook} and \cite{P},  the moduli scheme $\mathcal{A}^{\rm naive}_{\mathbf{K}}$ is connected to the naive local model ${\rm M}^{\rm naive}$ via the local model diagram 
\[
\begin{tikzcd}
&\wti{\mathcal{A}}^{\rm naive}_{\mathbf{K}}(G, X)\arrow[dl, "\psi_1"']\arrow[dr, "\psi_2"]  & \\
\mathcal{A}^{\rm naive}_{\mathbf{K}}  && {\rm M}^{\rm naive}
\end{tikzcd}
\]
where the morphism $\psi_1$ is a $\scrG$-torsor and $\psi_2$ is a smooth and $\scrG$-equivariant morphism. Here, $\scrG$ is the smooth Bruhat-Tits group scheme over $\text{Spec}(\mathbb{Z}_p)$ such
that $P_{\{m\}} = \scrG(\mathbb{Z}_p)$ and $\scrG \otimes_{\mathbb{Z}_p} \mathbb{Q}_p = G\otimes_{\mathbb{Q}} \mathbb{Q}_p $. In particular, since $\scrG$ is smooth, the above imply that $\mathcal{A}^{\rm naive}_{\mathbf{K}} $ is \'etale locally isomorphic to ${\rm M}^{\rm naive}$. From \cite[Remark 2.6.10]{PRS}, we have that the naive local model is never flat  and by the above the
same is true for $\mathcal{A}^{\rm naive}_{\mathbf{K}} $. Denote by $ \mathcal{A}^{\rm flat}_{\mathbf{K}} $ the flat closure of ${\rm Sh}_{\mathbf{K}}(G, X)\otimes_{E} E_v$ in $ \mathcal{A}^{\rm naive}_{\mathbf{K}}$. As in \cite{PR}, there is a relatively representable smooth morphism of relative dimension ${\rm dim} (G)$,
\[
\mathcal{A}^{\rm flat}_{\mathbf{K}} \to [\scrG \backslash \Mloc]
\]
where the local model $\Mloc $ is defined as the flat closure of $ {\rm M}^{\rm naive} \otimes_{\mathcal{O}} E_v$ in ${\rm M}^{\rm naive}$. Here we want to mention that from \cite[\S 8.2.4]{PRS} we have that $\Mloc $ agrees with the ``canonical", in the sense of Scholze-Weinstein, Pappas-Zhu {\color{blue} \cite{PZ}} local model $\mathbb{M}_{P_{\{m\}}}(G,\mu_{r,s})$ for the local model triple $(G,\mu_{r,s},P_{\{m\}})$; for the definition of the cocharacter $\mu_{r,s}$ see (\ref{cochar}).

We now consider the following variation $\mathcal{A}_{\mathbf{K}}$ of the above moduli problem of abelian schemes given in \cite[Definition 14.1]{PR2}. A point of $\mathcal{A}_{\mathbf{K}}$ with values in the $O_{K_1}$-scheme $S$ is the isomorphism class of the following set of data $(A,\bar{\lambda},\iota, \bar{\eta},\calG) $ such that
\begin{enumerate}
\item The quadruple $(A,\bar{\lambda},\iota, \bar{\eta})$ is an object of  $\mathcal{A}^{\rm naive}_{\mathbf{K}}(S)$.
\item	The subspace $\calG$ satisfying $ \calG \subset {\rm Fil}^0 (A) \subset H_{dR}^1(A)$ of rank $s$, with the \textit{splitting conditions}
\[
\quad O_{K_1} ~\text{acts by $\sigma_1$ on $\calG$, and by $\sigma_2$ on ${\rm Fil}^0 (A)/\calG$},
\]
where $\Gal(K_1/\QQ_p)=\{\sigma_1=id, \sigma_2\}$.
\end{enumerate}
Here we want to mention that the splitting conditions imply $(\iota(\pi)-\pi )\mathcal{G} = (0) $ and $(\iota(\pi)+\pi ){\rm Fil}^0 (A) \subset \calG$. As observed in \cite{ZacZhao1}, $\mathcal{A}_{\mathbf{K}}$ is never flat for any signature $(n-s,s)$. In loc. cit., the authors considered the closed subscheme $\mathcal{A}^{\rm spl}_{\mathbf{K}} \subset \mathcal{A}_{\mathbf{K}}$ defined by the
additional condition that 

\begin{enumerate}[resume]
\item (Spin Condition)  The rank of $(\iota(\pi)+\pi)$ on $ {\rm Fil}^0 (A)$ has the same parity as $s$.
%the line $\wedge^n({\rm Fil}^0 (A))$ is contained in \[ im[W(H_{dR}^1(A))_{(-1)^s}^{r,s}\otimes_{O_{K_1}}\calO_S\rightarrow W(H_{dR}^1(A))\otimes_{O_{K_1}}\calO_S].\]
\end{enumerate}
%{\color{blue}We refer to \cite{Sm2}, \cite{Sm3} for details. We claim that $\mathcal{A}^{\rm spl}_{\mathbf{K}}$ is the ``nice" integral model that we want, and} 
There is a forgetful morphism $\tau_1 :   \mathcal{A}^{\rm spl}_{\mathbf{K}} \longrightarrow \mathcal{A}^{\rm naive}_{\mathbf{K}}\otimes_{\mathcal{O}} O_{K_1}.$ Also, $ \mathcal{A}^{\rm spl}_{\mathbf{K}} $ has the same \'etale local structure as the splitting model $\Mspl$; it is a ``linear modification" of $\mathcal{A}^{\rm naive}_{\mathbf{K}}\otimes_{\mathcal{O}} O_{K_1}$ in the sense of \cite[\S 2]{P}. Note that there is also a corresponding forgetful morphism $\tau : \Mspl \longrightarrow  {\rm M}^{\rm naive} \otimes_{\mathcal{O}} O_{K_1}$.

The behavior of the splitting models depends heavily on the level subgroup and becomes more complicated for higher signatures. Let us briefly mention some results from the existing literature: A) For the self-dual case, Kr\"amer \cite{Kr} proved that $\Mspl$ is regular with semi-stable reduction but as the second author showed in \cite{Zac1} this is not true for the signature $(n-2,2)$. Also, in \cite{Zac1}, the author shows that the blow-up of the splitting model along a smooth (non Cartier) divisor produces a regular semistable model. For higher signatures, the work of the first author with Hernandez \cite{BH} shows that $\Mspl$ is always flat, normal with reduced special fiber but not semi-stable. B) For the almost $\pi$-modular case, Richarz \cite{Richarz} proved $\Mspl$ has semi-stable reduction for $(n-1,1)$. C) For the strongly non-special maximal cases the second and  third author \cite{ZacZhao2} showed that $\Mspl$ is flat, normal, and Cohen-Macaulay for $(n-1,1)$ and then they resolved the singularities by considering a certain blow-up. (For this case a different variation of a splitting model was studied more recently by He-Luo-Shi \cite{HLS}). Here we want to mention that for all these cases the splitting conditions imply the spin condition. However, the plot is thicker for the $\pi$-modular case that we consider in this paper since, as we mentioned above, the splitting model without the spin condition is never flat.
%For the strongly non-special maximal cases, the second and the third author showed that $\Mspl$ is flat, normal, and Cohen-Macaulay for $(n-1,1)$. 
%For all the above cases except the $ \pi$-modular case, the splitting conditions imply the spin condition. For $\pi$-modular case, we showed that the splitting condition does not imply the spin condiiton for any signature in \cite{ZacZhao1}, so that the splitting models we concerned in this paper are more complicated.

\subsection{} The main result of the present paper is the following theorem.
%The main result of the paper is the following theorem: 

 \begin{Theorem}
For any signature $ (n-s,s) $ and for every $K^p$ as above, there is a scheme $\mathcal{A}^{\rm spl}_{\mathbf{K}}$, flat over $\Spec(O_{K_1})$, 
 with
 \[
\mathcal{A}^{\rm spl}_{\mathbf{K}}\otimes_{O_{K_1}} K_1={\rm Sh}_{\mathbf{K}}(G, X)\otimes_{E} K_1,
 \]
 and which supports a local model diagram
  \begin{equation}\label{LMdiagramRegIntro}
\begin{tikzcd}
&\wti{\mathcal{A}}^{\rm spl}_{\mathbf{K}}(G, X)\arrow[dl, "\pi^{\rm reg}_K"']\arrow[dr, "q^{\rm reg}_K"]  & \\
\mathcal{A}^{\rm spl}_{\mathbf{K}}  &&  {\rm M}^{\rm spl}
\end{tikzcd}
\end{equation}
such that:
\begin{itemize}
\item[a)] $\pi^{\rm reg}_{\mathbf{K}}$ is a $\scrG$-torsor for the parahoric group scheme $\scrG$ that corresponds to $P_{\{m\}}$.

\item[b)] $q^{\rm reg}_{\mathbf{K}}$ is smooth and $\scrG$-equivariant.

\item[c)] $\mathcal{A}^{\rm spl}_{\mathbf{K}}$ is %$O_{K_1}$-flat, 
normal, Cohen-Macaulay and its special fiber is reduced. %Also, $\mathcal{A}^{\rm spl}_{\mathbf{K}}$ has an explicit moduli-theoretic description.
\end{itemize}
 \end{Theorem}
Note that by the above $\mathcal{A}^{\rm spl}_{\mathbf{K}}$ has an explicit moduli-theoretic description. Since every point of $\mathcal{A}^{\rm spl}_{\mathbf{K}} $ has an \'etale neighborhood which is also \'etale over ${\rm M}^{\rm spl}$, it suffices to show that ${\rm M}^{\rm spl}$ has the above nice properties. In particular it's enough to study the special fiber $\overline{\mathrm{M}}^{\mathrm{spl}}$ of ${\rm M}^{\rm spl}$. We do this by putting a certain combinatorial stratification described below.

First let's consider the special fiber $\overline{\mathcal{A}}^{\mathrm{spl}}_{\mathbf{K}}$ of $\mathcal{A}^{\rm spl}_{\mathbf{K}}$: recall that the subspace $ \calG \subset {\rm Fil}^0 (A)$ has rank $s$. Also, there is a second direct factor $ \calG'  \subset {\rm Fil}^0 (A)$ of rank $r$, which is obtained
from $\calG$ and the polarization $\bar{\lambda}$ (see \S \ref{Splitting} for more details). Define the subspace $A_{h,l} = \{x\in \overline{\mathcal{A}}^{\mathrm{spl}}_{\mathbf{K}}| \, \text{dim } (\iota(\pi) ){\rm Fil}^0 (A)  =h ,\, \text{dim } \calG \cap \calG'=l\}$ of $\overline{\mathcal{A}}^{\mathrm{spl}}_{\mathbf{K}}$. Using the local model diagram, $A_{h,l}$ corresponds to the subspace $X_{h,l}$ of $\overline{\mathrm{M}}^{\mathrm{spl}}$ (see Definition \ref{defhl}). 
Now, we are ready to define a stratification on $\overline{\mathrm{M}}^{\mathrm{spl}}$. 
\begin{Proposition}
Let $(h,l)$ be integers with $ 0\leq h\leq l\leq s$ and $h,l  = s \mod 2$. Then
\[
\overline{\mathrm{M}}^{\mathrm{spl}} = \coprod_{\substack{0\leq h\leq l\leq s \\ h,l \, \equiv \, s \, \text{mod} \, 2}} X_{h,l}.
\]
Each stratum $X_{h,l}$ is nonempty and smooth, and is
equidimensional of dimension $ rs - \frac{(l-h)(l-h-1)}{2}$.
\end{Proposition}
Moreover we have the closure relations
\begin{Proposition}
Let $(h,l)$ be integers with $ 0\leq h\leq l\leq s$ and $h,l  = s \mod 2$. Then
\[
\overline{X_{h,l}} = \coprod_{\substack{0\leq  h' \leq h\leq l \leq l' \leq s \\ h',l' \, \equiv \, s \, \text{mod} \, 2 }} X_{h',l'}.
\]
\end{Proposition}
From the above we can deduce that the irreducible components of the scheme $\overline{\mathrm{M}}^{\mathrm{spl}} $ are the closures of $X_{h,h}$ for $h  = s \mod 2$ and $0 \leq h \leq s$. Also, our computations show that  ${\rm M}^{\rm spl}$ is not semi-stable when $s > 3$. We note that the cases $s \leq 3$ have already been addressed in \cite{ZacZhao1} using different techniques, where the authors proved that ${\rm M}^{\rm spl}$ is smooth for the signature $(n-1,1)$ and is semi-stable for the signatures $(n-2,2)$ and $(n-3,3)$.
%%%We do not know the Richarz case for higher signatures.
%{\color{blue}By the work of \cite{Kr}, \cite{Richarz}, \cite{Zac1}, \cite{ZacZhao1}, \cite{ZacZhao2}, and this paper, we completely solved the semi-stable problem for splitting models for unitary Shimura varieties:

%\begin{Theorem}
%The splitting model $\rmM^{\mathrm{spl}}$ has a semi-stable reduction if and only if the following list:
%\begin{itemize}
%	\item the dimension $n$ is odd, level subgroup $I=\{0\}$, and the signature $(r,s)=(n-1,1)$;
%	\item the dimension $n$ is odd, level subgroup $I=\{\lfloor n/2\rfloor\}$, and the signature $(r,s)=(n-1,1)$;
%	\item the dimension $n$ is even, level subgroup $I=\{n/2\}$, and the signature $(r,s)$ satisfying $s\leq 3$.
%\end{itemize}
%\end{Theorem}
%}
Lastly, as we mentioned above, a similar stratification was previously used in \cite{BH} to study the splitting model for the level subgroup $I= \{0\}$. In Section \ref{appl}, we consider both level structures, i.e. $I = \{m\}$ and $I= \{0\}$. First, we compare the stratification of the splitting models given by $X_{h,l}$ with the stratification of the local models defined by the Schubert cells. Then, we build on this to produce some affine Demazure resolutions of certain Schubert varieties.
\smallskip

{\bf Acknowledgements:} We thank J. Louren\c{c}o and G. Pappas for useful comments on a preliminary version of this article. The second author would like to thank École Polytechnique for its hospitality during September 2024, during which part of this work was done. Also, the second author was supported by CRC 1442 Geometry: Deformations and Rigidity and the Excellence Cluster of the Universität Münster.

\section{Integral models of unitary Shimura varieties}\label{Integral}

In this section, we construct $p$-adic integral models of Shimura varieties in the case of a quasi-split unitary group with a special maximal parahoric level structure.

\subsection{Unitary Shimura varieties}\label{UnitaryShimura}
First we fix some notation. Let $K$ be an imaginary quadratic field with an embedding $\varepsilon: K\rightarrow \mathbb{C}$. Denote by $a\mapsto \bar{a}$ the non-trivial automorphism of $K$. Let $W$ be a $n$-dimensional $K$-vector space, equipped with a non-degenerate hermitian form $\phi$. Assume $n\geq 3$ and set $W_{\mathbb{C}}=W\otimes_{K,\varepsilon}\CC$. Choose a suitable isomorphism $W_{\mathbb{C}}\simeq \CC^n$, the non-degenerate hermitian form $\phi$ can be expressed as $\phi(x,y)=\bar{x}^tBy$, where 
\[
B={\rm diag}((-1)^{(s)}, 1^{(r)})
\]
(We write $a^{(m)}$ to denote a list of $m$ copies of $a$). We will say that $\phi$ has signature $(r,s)$ with $r+s=n$. By replacing $\phi$ to $-\phi$ if necessary, we can assume that $s\leq r$.  To give a complex structure on $W\otimes_{\QQ} \RR=W_{\CC}$, let $J: W_{\CC}\rightarrow W_{\CC}$ be the endomorphism given by the matrix $-\sqrt{-1}H$. We have $J^2 =-id$ and the endomorphism $J$ gives an $\RR$-algebra homomorphism $h_0: \CC\rightarrow \End_{\RR}(W\otimes_{\QQ} \RR)$ with $h_0(\sqrt{-1})=J$. 

Consider the group $G = GU_n$ of unitary similitudes for $(W,\phi)$ of dimension $n\geq 3$ over $\QQ$:
\[
G(R):=\{g\in \GL_{K}(W\otimes_{\QQ} R)\mid \phi(gx, gy)=c(g)\phi(x,y), \quad c(g)\in R^\times \}.
\]
We define a homomorphism $h: \Res_{\CC/\RR} \GG_{m,\CC}\rightarrow G_{\RR}$ by restricting $h_0$ to $\CC^\times$. Since $G\otimes \CC\simeq \GL_{n, \CC}\times \GG_{m,\CC}$, the image of $h_{\CC}(z, 1): \CC^\times\rightarrow G(\CC)\simeq \GL_n(\CC)\times \CC^\times$ is given by
\begin{equation}\label{cochar}
\mu_{r,s}(z)={\rm diag}(z^{(s)}, 1^{(r)}, z),
\end{equation}
up to conjugation. This is a cocharacter of $G$ defined over the number field $E$.

Let $X_h$ be the conjugation orbit of $h(i)$ under $G(\RR)$. The pair $(G, X_h)$ gives rise to a Shimura variety $Sh(G,X_h)$ over the reflex field $E$. (See \cite[\S 1.1]{PR} and \cite[\S 3]{P} for more details on the description of the unitary Shimura varieties.) Let $p$ be an odd prime which ramifies in $K$. We set $\mathcal{O} = O_{E_v}$ where $v$ the unique prime ideal of $E$ above $(p)$. Choose also a sufficiently small compact open subgroup $K^p$ of the prime-to-$p$ finite adelic points $G({\mathbb A}_{f}^p)$ of $G$ and a parahoric subgroup $K_p\subset G(\QQ_p)$. Set $\mathbf{K}=K^pK_p$. The Shimura variety  ${\rm Sh}_{\mathbf{K}}(G, X_h)$ with complex points
 \[
 {\rm Sh}_{\mathbf{K}}(G, X_h)(\mathbb{C})=G(\mathbb{Q})\backslash X_h\times G({\mathbb A}_{f})/\mathbf{K}
 \]
is of PEL type and has a canonical model over the reflex field $E$.

\subsection{Integral models of unitary Shimura varieties}\label{Int} 
It turns out that the parahoric subgroups $K_p$ are the neutral components of the subgroups of $G$ that stabilize certain sets of lattices. Set $K_1=K\otimes_{\QQ} \QQ_p$ and let $\pi$ be a uniformizer with $\pi^2=p$. Consider $V=W\otimes_{\QQ} \QQ_p$. We assume that the hermitian form $\phi$ is split on $V$, i.e. there is a basis $e_1, \dots, e_n$ such that $\phi(e_i, e_{n+1-j})=\delta_{ij}$ for $1\leq i, j\leq n$. We denote by 
\[
\Lambda_i = \text{span}_{O_{K_1}} \{\pi^{-1}e_1, \dots, \pi^{-1}e_i, e_{i+1}, \dots, e_n\}
\]
the standard lattices in $V$.

Now let $I$ be a non-empty subset of $\{0, \cdots, \lfloor \frac{n}{2} \rfloor \}$. The stabilizer subgroup 
\[
P_I:=\{g\in G(\QQ_p) \mid g\cdot\Lambda_i=\Lambda_i , \quad \forall i \in I\}.
\]
of $G(\QQ_p)$ is not a parahoric subgroup in
general, since it may contain elements with nontrivial Kottwitz invariant. (See \cite[\S 1.2]{PR} or \cite[\S 1.2]{PR3}.) However, since we are only interested in the special maximal parahoric subgroups in this paper, it turns out that these subgroups are indeed the stabilizer subgroups for some $I$.

By \cite{PR}, the special maximal parahoric subgroups of $GU_n$ depends on the parity of the dimension $n$. When $n=2m+1$ is odd, there are two special maximal parahoric subgroups up to conjugate, corresponding to $I=\{0\}$ and $I=\{m\}$. For $I=\{0\}$, integral models for different signatures have been studied by the works of Pappas \cite{P}, Kr\"amer \cite{Kr}, Zachos \cite{Zac1} and Bijakowski-Hernandez \cite{BH}. In the case of $I=\{m\}$, integral models for the signature $(n-1,1)$ have been studied by Arzdorf \cite{A} and Richarz \cite{Richarz}.

In this paper, we will focus on the special maximal parahoric subgroups when $n=2m$ is even. From \cite{PR}, we have that, in the even dimensional case, there is a unique special maximal parahoric subgroup up to conjugation, corresponding to $I=\{m\}$. From now on, we fix $n=2m$, $I=\{m\}$ and the parahoric subgroup $K_p=P_{\{m\}}$ such that $\mathbf{K}=K^pP_{\{m\}}$.

Let us describe the naive integral model $\mathcal{A}^{\rm naive}_{\mathbf{K}}$. We let $\mathcal{L}$ be the self-dual multichain consisting of lattices $\{\Lambda_j\}_{j\in n\mathbb{Z} \pm m}$. Let $\scrG = \underline{{\rm Aut}}(\mathcal{L})$ be the smooth group scheme over $\mathbb{Z}_p$ with $P_{\{m\}} = \scrG(\mathbb{Z}_p)$ the subgroup of $G(\mathbb{Q}_p)$ fixing the lattice chain $\mathcal{L}$. Consider the moduli functor $\mathcal{A}^{\rm naive}_{\mathbf{K}}$ over $\Spec \mathcal{O} $ given in \cite[Definition 6.9]{RZbook}: A point of $\mathcal{A}^{\rm naive}_{\mathbf{K}}$ with values in the $\Spec \mathcal{O} $-scheme $S$ is the isomorphism class of quadruples $(A,\iota, \bar{\lambda}, \bar{\eta})$ consisting of:
\begin{enumerate}
    \item An object $(A,\iota)$, where $A$ is an abelian scheme {\color{blue} of} relative dimension $n$ over $S$ (terminology
of \cite{RZbook}), compatibly endowed with an
action of $\calO$: 
\[ \iota: \calO \rightarrow \text{End} \,A \otimes \mathbb{Z}_p.\] 
    \item A $\mathbb{Q}$-homogeneous principal polarization $\bar{\lambda}$ of the $\mathcal{L}$-set $A$.
    \item A $K^p$-level structure
    \[
\bar{\eta} : H_1 (A, {\mathbb A}_{f}^p) \simeq W \otimes  {\mathbb A}_{f}^p \, \text{ mod} \, K^p
    \]
which respects the bilinear forms on both sides up to a constant in $({\mathbb A}_{f}^p)^{\times}$ (see loc. cit. for
details).\\
The set $A$ should satisfy the determinant condition (i) of loc. cit.
\end{enumerate}

For the definitions of the terms employed here we refer to loc.cit., 6.3–6.8 and \cite[\S 3]{P}. The functor $\mathcal{A}^{\rm naive}_{\mathbf{K}}$ is representable by a quasi-projective scheme over $\mathcal{O}$. Since the Hasse principle is satisfied for the unitary group, we can see as in loc. cit. that there is a natural isomorphism
\[
\mathcal{A}^{\rm naive}_{\mathbf{K}} \otimes_{\calO} E_v = {\rm Sh}_{\mathbf{K}}(G, X_h)\otimes_{E} E_v.
\]
The naive integral model $\mathcal{A}^{\rm naive}_{\mathbf{K}}$ is not flat. Denote by $ \mathcal{A}^{\rm flat}_{\mathbf{K}} $ the flat closure of ${\rm Sh}_{\mathbf{K}}(G, X_h)\otimes_{E} E_v$ in $ \mathcal{A}^{\rm naive}_{\mathbf{K}}$. One can now consider a variation of the moduli of abelian schemes $\mathcal{A}^{\rm spl}_{\mathbf{K}}$ by adding in the moduli problem an additional subspace in the Hodge filtration $ {\rm Fil}^0 (A) \subset H_{dR}^1(A)$ of the universal abelian variety $A$ with some restricting properties. A point of $\mathcal{A}^{\rm spl}_{\mathbf{K}}$ with values in the $O_{K_1}$-scheme $S$ is the isomorphism class of the following set of data $(A,\iota, \bar{\lambda}, \bar{\eta},\calG) $ such that
\begin{enumerate}
\item[(a).] The quadruple $(A,\iota,\bar{\lambda}, \bar{\eta})$ is an object of  $\mathcal{A}^{\rm naive}_{\mathbf{K}}(S)$.
\item[(b).]	The subspace $\calG$ satisfying $ \calG \subset {\rm Fil}^0 (A) \subset H_{dR}^1(A)$ of rank $s$, with 
\begin{equation}\label{eq 803}
O_{K_1} ~\text{acts by $\sigma_1$ on $\calG$, and by $\sigma_2$ on ${\rm Fil}^0 (A)/\calG$},
\end{equation}
where $\Gal(K_1/\QQ_p)=\{\sigma_1=id, \sigma_2\}$.
%\item[(c).] {\color{blue} The line $\wedge^n({\rm Fil}^0 (A))$ is contained in
%\[
%im[W(H_{dR}^1(A))_{(-1)^s}^{r,s}\otimes_{O_{K_1}}\calO_S\rightarrow W(H_{dR}^1(A))\otimes_{O_{K_1}}\calO_S].
%\]
\item[(c).] The rank of $(\iota(\pi)+\pi)$ on $ {\rm Fil}^0 (A)$ has the same parity as $s$.
\end{enumerate}

\begin{Remarks}{\rm 
\begin{enumerate}
    \item We call condition (b) the {\it splitting condition} and (c) the {\it spin condition}. Condition (c) imitates the definition of the spin condition (4') for the corresponding splitting model (see Definition \ref{DefSplSpin}); we refer the reader to \S \ref{spl} where the spin condition is defined and discussed for the splitting model.%It is not hard to see that  condition (c) is independent of the choice of $\pi$, so that it is a well-defined condition. We refer to \cite[\S 5.3]{Sm3} for more details. 
    \item We note that in the self-dual case, $I= \{0\}$, the conditions (a) and (b) are enough to give you a flat integral model; see \cite{BH} and \cite{Zac1}. However, this does not hold in our case, the $\pi$-modular case i.e. where $n$ is even and $I= \{m\}$, as shown in \cite[\S 5, \S 7]{ZacZhao1}.  
%    \item Lastly, we want to mention that in our setting, {\color{red} the so-called ``spin condition" amounts to condition (c). This imitates the definition of the spin condition (4') for the corresponding splitting model; we refer the reader to Definition \ref{DefSplSpin} where the spin condition is defined and discussed for the splitting model.}
% \begin{enumerate} \item[(c').] The rank of image of $\iota(\pi)+\pi: {\rm Fil}^0 (A)\rightarrow {\rm Fil}^0 (A)$ has the same parity as $s$.\end{enumerate}as shown in \cite[Remark 9.9]{RSZ}. 
 %(Note that they assumed $r\leq s$ in \cite{RSZ}.)} In general, the spin condition is more complicated and for its formulation we refer the reader to the works of Pappas-Rapoport \cite[\S 7]{PR} and Smithling \cite{Sm2}, \cite{Sm3}. 
\end{enumerate}
}\end{Remarks}

There is a forgetful morphism
\[
\tau_1 :   \mathcal{A}^{\rm spl}_{\mathbf{K}} \longrightarrow \mathcal{A}^{\rm naive}_{\mathbf{K}}\otimes_{\mathcal{O}} O_{K_1}
\]
defined by $(A,\iota, \bar{\lambda}, \bar{\eta},\calG) \mapsto (A,\iota, \bar{\lambda}, \bar{\eta})$. In this paper, we will prove that $\mathcal{A}^{\rm spl}_{\mathbf{K}}$ is flat for any signature $(r,s)$. Thus, the $\mathcal{A}^{\rm flat}_{\mathbf{K}}$ is the scheme-theoretic image of $\tau_1$.

\section{Splitting models of unitary Shimura varieties}\label{Splitting}

\subsection{Local model diagrams}\label{localdiagram}

The singularities of $\mathcal{A}^{\rm naive}_{\mathbf{K}}$ are modeled by the {\it naive local model} ${\rm M}^{\rm naive}$ in \cite{RZbook}. As is explained in \cite{P}, the naive local model  is connected to the moduli scheme $\mathcal{A}^{\rm naive}_{\mathbf{K}}$ via the local model diagram 
\[
\mathcal{A}^{\rm naive}_{\mathbf{K}} \ \xleftarrow{\psi_1} \Tilde{\mathcal{A}}^{\rm naive}_{\mathbf{K}} (G,X_h) \xrightarrow{\psi_2} {\rm M}^{\rm naive}
\]
where the morphism $\psi_1$ is a $\scrG$-torsor and $\psi_2$ is a smooth and $\scrG$-equivariant morphism. Therefore, there is a relatively representable smooth morphism
 \[
 \mathcal{A}^{\rm naive}_{\mathbf{K}} \to [\scrG \backslash  {\rm M}^{\rm naive}]
 \]
where the target is the quotient algebraic stack. The closed subscheme $ \mathcal{A}^{\rm flat}_{\mathbf{K}} $ of $ \mathcal{A}^{\rm naive}_{\mathbf{K}}$ is a linear modification of $ \mathcal{A}^{\rm naive}_{\mathbf{K}}$ in the sense of \cite[\S 2]{P}. In particular, as in \cite{PR}, there is a relatively representable smooth morphism of relative dimension ${\rm dim} (G)$,
\[\mathcal{A}^{\rm flat}_{\mathbf{K}} \to [\scrG \backslash \Mloc].\]
This of course implies that $\mathcal{A}^{\rm flat}_{\mathbf{K}}$ is \'etale locally isomorphic to the {\it local model} $\Mloc$.
%and $ \mathcal{A}^{\rm flat}_{\mathbf{K}} $ has the same \'etale local structure as the {\it local model} $\Mloc $.

Similarly, $\mathcal{A}^{\rm spl}_{\mathbf{K}}$ is a linear modification of $ \mathcal{A}^{\rm naive}_{\mathbf{K}} \otimes_{\calO}O_{K_1}$ and has the same \'etale local structure as the {\it splitting model} $\Mspl$ (see also \cite[\S 15]{PR2}). 
%Similar to the definition of $\mathcal{A}^{\rm spl}_{\mathbf{K}}$,  we can define that the {\it local model} $\Mloc $ is the flat closure of $ {\rm M}^{\rm naive} \otimes_{\mathcal{O}} E_v$ in ${\rm M}^{\rm naive}$. By the above we can see, as in \cite{PR}, that there is a relatively representable smooth morphism of relative dimension ${\rm dim} (G)$, \[\mathcal{A}^{\rm flat}_{\mathbf{K}} \to [\scrG \backslash \Mloc].\] This of course implies that $\mathcal{A}^{\rm flat}_{\mathbf{K}}$ is \'etale locally isomorphic to the local model $\Mloc$. By using the local model diagram,   $\mathcal{A}^{\rm spl}_{\mathbf{K}}$ has the same ´etale local structure as $\Mspl$; it is a linear modification of $\Mloc \otimes_{calO} O_{K_1}$ in the sense of $\mathcal{A}^{\rm flat}_{\mathbf{K}}$ can induce a resolution $\Mspl$ of $\Mloc$, such that $\mathcal{A}^{\rm spl}_{\mathbf{K}}$ has the same \'etale local structure as $\Mspl$. We call $\Mspl$ the {\it splitting model} over $O_{K_1}$. It is a ``linear modification" of ${\rm M}^{\rm naive}\otimes_{\mathcal{O}} O_{K_1}$ in the sense of \cite[\S 2]{P} (see also \cite[\S 15]{PR2}). 
In Section \ref{Flatness}, we will show that $\Mspl$ is flat for any signature $(r,s)$ which in turn implies that $\mathcal{A}^{\rm spl}_{\mathbf{K}}$ is also flat.

\subsection{Local Models and Variants}\label{spl}  In this section, we will give the explicit definitions of the local models and their variants that we are interested in. % define the  define the naive local model $\Mloc$ and the splitting model $\Mspl$ in a slightly more general context.

We use the notation of \cite{PR}. Let $F_0$ be a complete discretely valued field with ring of integers $O_{F_0}$, perfect residue field $k$ of characteristic $\neq 2$, and uniformizer $\pi_0$. Let $F/F_0$ be a ramified quadratic extension and $\pi \in F$ a uniformizer with $\pi^2 = \pi_0$.  Let $V$ be a $F$-vector space of dimension $n =2m> 3$ and let 
\[
\phi: V \times V \rightarrow F
\]
be an $F/F_0$-hermitian form. We assume that $\phi$ is split. This means that there exists a basis $e_1, \dots , e_n$ of $V$ such that 
\[
\phi(e_i,e_{n+1-j}) = \delta_{i,j} \quad \text{for  all} \quad i,j = 1, \dots, n.
\]
We attach to $\phi$ the respective symmetric $F_0$-bilinear form $(\ ,\ ): V \times V \rightarrow F_0$ given by 
\[
( x, y ) =  \frac{1}{2}\text{Tr}_{F /F_0} (\phi(x, y)).
\]
For any $O_F$-lattice $\Lambda$ in $V$, we denote by
\[
\Lambda^\vee = \{v \in V | ( v, \Lambda ) \subset O_{F_0} \}
\]
the dual lattice with respect to the symmetric form. The uniformizing element $\pi$ induces a $O_{F_0}$- linear mapping on $\Lambda$ which we denote by $t$. Similar to \S \ref{Int}, we can define the standard lattices
\[
\Lambda_i = \text{span}_{O_F} \{\pi^{-1}e_1, \dots, \pi^{-1}e_i, e_{i+1}, \dots, e_n\},
\]
for $i= 0, \cdots, n- 1$, and complete the $\Lambda_i$  to a self-dual periodic lattice chain $(\Lambda_j)_{j\in \ZZ}$ by including all the $\pi$-multiples, i.e., $\Lambda_{j} = \pi^{-k}\Lambda_i$ for $j \in \mathbb{Z}$ of the form $j = k \cdot n+i$ with $i =0, \cdots, n- 1$. 

 The symmetric form $ (\ , \ ) $ induce a perfect $O_{F_0}$-bilinear pairings
\begin{equation}\label{perfectpairing}
     \Lambda \times \Lambda^\vee \xrightarrow{ (\ , \ )} O_{F_0}
\end{equation}
for all $\Lambda$. Note that $\Lambda_i^\vee=\Lambda_{-i+n}$. Therefore, when $I=\{m\}$, $\Lambda_m$ is a self-dual lattice with respect to the form $ (\ , \ ) $.

Consider the unitary similitude group 
\[
G:=\{g\in \GL_{F}(V)\mid \phi(gx, gy)=c(g)\phi(x,y), \quad c(g)\in F_0^\times \},
\]
and the cocharacter $\mu_{r,s}=(1^{(s)}, 0^{(r)},1)$ of $D\times \GG_m\subset \GL_{n,F}\times \GG_{m,F}\simeq G\otimes F$, where $D$ is the standard maximal torus of diagonal matrices in $\GL_n$. (See \cite{Sm2} for more details.) We denote by $E$ the reflex field of $\{ \mu_{r,s}\}$; then $E = F_0$ if $r = s$ and $E = F$ otherwise (see \cite[\S 1.1]{PR}). We set $O := O_E$. By fixing the parahoric subgroup $K_p=P_{\{m\}}$, the cocharacter $\mu=\mu_{r,s}$, we can now define the naive local model ${\rm M}^{\rm naive}$:

\begin{Definition}\label{naivelocal}
${\rm M}^{\rm naive}$ is the functor which associates to each scheme $S$
over $\Spec O$ the set of subsheaves $\mathcal{F}_m$ of $O \otimes \mathcal{O}_S $-modules of $ \Lambda_m \otimes \mathcal{O}_S  $
such that
\begin{enumerate}
    \item $\mathcal{F}_m$ as an $\mathcal{O}_S $-module is Zariski locally on $S$ a direct summand of rank $n$;
    \item $\mathcal{F}_m$ is totally isotropic for $( \,
, \, ) \otimes \mathcal{O}_S$;
    \item (Kottwitz condition) $\text{char}_{t |  \mathcal{F}_m } (T)= (T + \pi)^r(T - \pi)^s $.
\end{enumerate}	
\end{Definition}

In \cite[\S 1.5.1]{PR}, the authors define the naive local model $\rmM^{\rm naive}$ that sends each $O$-scheme $S$ to the families of $O\otimes \calO_S$-modules $(\calF_i\subset \Lambda_i\otimes\calO_S)_{i \in n\mathbb{Z}\pm I}$ that satisfy the conditions (a)-(d) of loc. cit. Here our conditions (1)-(3) are equivalent to the conditions (a)-(d). (See \cite[\S 3]{ZacZhao1} for more details).

The functor ${\rm M}^{\rm naive}$ is represented by a closed subscheme, which we again
denote ${\rm M}^{\rm naive}$, of $\Gr(n, 2n) \otimes \Spec O$; hence ${\rm M}^{\rm naive}$ is a projective $ O$-scheme. (Here we denote by $\Gr(n, d)$ the Grassmannian scheme parameterizing locally direct summands of rank $n$ of a free module of rank $d$.) Also, ${\rm M}^{\rm naive}$ supports an action of $\scrG$. 

\begin{Proposition}\label{notflat}
We have 
\[
{\rm M}^{\rm naive} \otimes_{ O} E \cong \Gr(s,n)\otimes E.
\]
In particular, the generic fiber of ${\rm M}^{\rm naive}$ is smooth and geometrically irreducible of dimension $rs$.
\end{Proposition}
\begin{proof}
See \cite[\S 1.5.3]{PR}. 
    %and \cite[Proposition 2.2]{Kr}. 
\end{proof}

It was observed by \cite{P} that ${\rm M}^{\rm naive}$ fails to be flat in general. The main reason is that $F/F_0$ is a ramified extension and $t=\pi\otimes 1\in O_F\otimes_{O_{F_0}}\calO_S$ is nilpotent, such that the Kottwitz condition fails to impose a condition on the reduced special fiber. Let $\Mloc$ be the scheme theoretic closure of the generic fiber $ {\rm M}^{\rm naive} \otimes_{O} E$ in ${\rm M}^{\rm naive}$. The scheme $\Mloc$ is called the \textit{local model}. We have closed immersions of projective schemes
$
\Mloc \subset {\rm M}^{\rm naive}
$
which are equalities on generic fibers (see \cite[\S 1.5]{PR} for more details).
\begin{Proposition}
a) For any signature $(r, s)$, the special fiber of $\Mloc$ is integral and normal and has only rational singularities.

b) For $(r, s)= (n-1,1)$, $\Mloc$ is smooth. 
\end{Proposition}
\begin{proof}
See \cite[Theorem 5.1]{PR} and \cite[\S 5.3]{PR}. 
\end{proof} 

Recall that $t=\pi\otimes 1, \pi=1\otimes \pi$ in $ O_F\otimes_{O_{F_0}}\calO_S$ and set $ W= \wedge^{n}_F (V\otimes_{F_0}F).$ The symmetric form $( \, , \, )$ splits over $V$ and thus there is a canonical decomposition
\[
W = W_1 \oplus W_{-1}
\]
of $W $ as an $SO_{2n}(F)$-representation (see \cite[\S 7]{RSZ}). For a standard lattice $\Lambda_i$ in $V$, we set $W(\Lambda_i)=\wedge^n(\Lambda_i\otimes_{O_{F_0}}O_F )$ and $W(\Lambda_i)_{\pm 1 } = W_{\pm 1} \cap W(\Lambda).$ For an $O_F$-scheme $S$ and $\epsilon \in \{\pm1\}$, we set 
\[
L_{i,\epsilon}(S)=\text{im}[ W(\Lambda_i)_{\epsilon}\otimes_{O_F} \mathcal{O}_S  \longrightarrow  W(\Lambda)\otimes_{O_F} \mathcal{O}_S ].
\]

%We now formulate the spin condition. 
Now we are ready to define the splitting model $\Mspl$. 
%As we mentioned in \S \ref{localdiagram}, there is a resolution $\Mspl$ of $\Mloc$:

\begin{Definition}\label{DefSplSpin}
Let $\Mspl$ be the moduli space over $O_F$, whose points of an $O_F$-scheme $S$ are the set of pairs $(\calF_m,\calG_m)$, where $\calF_m$ (resp. $\calG_m$) are Zariski locally $\calO_S$-direct summands of rank $n$ (resp. $s$), such that
\begin{enumerate}
    \item $ 0 \subset \calG_m \subset \calF_m \subset \Lambda_m  \otimes \mathcal{O}_S$;
    \item $\calF_m = \calF_m^{\bot}$, i.e. $\calF_m$ is totally isotropic for $( \,
, \, ) \otimes \mathcal{O}_S$;
    \item ({\rm splitting condition}) $(t + \pi) \calF_m \subset \calG_m,\quad (t - \pi)\calG_m = 0$; 
    %\item ({\rm spin condition}). The rank of $(t + \pi)$ on $ \calF_m$ has the same parity as $s$.
    \item ({\rm spin condition}). The line bundle $\wedge^n \calF_m \subset W(\Lambda_m)\otimes_{O_F} \mathcal{O}_S$ is contained in $ L_{m,(-1)^s}(S)$. 
\end{enumerate}
\end{Definition}

It is easy to see that condition (3) is the equivalent condition (b) in $\mathcal{A}^{\rm spl}_{\mathbf{K}}$. Also, in our setting, from \cite[Remark 9.9]{RSZ} (see also \cite[\S 7]{ZacZhao1}), we get that the spin condition can be characterized as the following condition:

(4')  The rank of $(t + \pi)$ on $ \calF_m$ has the same parity as $s$.

(Note that, in \cite{RSZ},  the authors assume $r\leq s$.) %In general, the spin condition is more complicated and for its formulation we refer the reader to the works of Pappas-Rapoport \cite[\S 7]{PR} and Smithling \cite{Sm2}, \cite{Sm3}. 
The functor $\Mspl$ is represented by a projective $O_F $-scheme and there is a $\scrG$-equivariant projective morphism 
\begin{equation}\label{formor}
\tau : \Mspl \rightarrow  {\rm M}^{\rm naive}\otimes_{O} O_F	
\end{equation}
which is given by $(\mathcal{F}_m,\mathcal{G}_m) \mapsto \mathcal{F}_m$ on $S$-valued points. (Indeed, we can easily see, as in \cite[Definition 4.1]{Kr}, that $\tau$ is well defined.) The morphism $\tau$ induces an isomorphism on the generic fibers (see \cite[Remark 4.2]{Kr}).

In \cite{ZacZhao1}, the authors showed that $\Mspl$ is smooth for the signature $(n-1,1) $ and has semi-stable reduction for the signatures $(n-2,2)$ and $(n-3,3)$. In our paper, we will show that 

\begin{Theorem}\label{mainflat}
For any signature $(r,s)$, the splitting model $\Mspl$ is flat, normal and Cohen-Macaulay over $O_F$. The special fiber $\Mspl\otimes k$ is reduced and has $\frac{s}{2}$ irreducible components when $s$ is even, and $\frac{s+1}{2}$ irreducible components when $s$ is odd.
\end{Theorem}
The proof of the above theorem will be carried out in \S \ref{Closure} and \S \ref{Flatness}. In particular, we give a description of the irreducible components
of the special fiber in \S \ref{Closure} and we prove the rest of statements in \S \ref{Flatness}.

\begin{Corollary}\label{flatintegral}
The integral model $\mathcal{A}^{\rm spl}_{\mathbf{K}}$ is a flat normal, and Cohen-Macaulay scheme with reduced special fiber.
\end{Corollary}

From the above we deduce that the $\scrG$-equivariant projective morphism $ \tau$ factors through $ \Mloc \otimes_{O} O_F \subset {\rm M}^{\rm naive}\otimes_{O} O_F	$ because of flatness; the generic
fiber of all of these is the same.

\subsection{Geometry of the Special fiber of $\Mspl$}
In this section, we will show that the special fiber of $\Mspl$ is stratified by an explicit poset with a combinatorial description. Our work is similar to the work of \cite{BH}, except that we consider the special maximal parahoric $K_p=P_{\{m\}}$ instead of $P_{\{0\}}$.

Consider the standard lattice $\Lambda_m$:
\[
\begin{array}{ll}
\Lambda_m &= \text{span}_{O_F} \{\pi^{-1}e_1, \dots, \pi^{-1}e_m, e_{m+1}, \dots, e_n\}\\	
&=\text{span}_{O_{F_0}} \{\pi^{-1}e_1, \dots, \pi^{-1}e_m, e_{m+1}, \dots, e_n,e_1, \dots, e_m, \pi e_{m+1}, \dots, \pi e_n\}.
\end{array}
\]
The symmetric matrix of $(\ ,\ )$ in this basis is 
\[
\left[\ 
\begin{matrix}[c|c]
0_n & J_n  \\ \hline
-J_n & 0_n 
\end{matrix}\ \right], \,\,\, \text{where} \,\, J_n = \left[\ 
\begin{matrix}[c|c]
0_m & -H_m  \\ \hline
H_m & 0_m 
\end{matrix}\ \right]
\]
and $H_m$ is the unit antidiagonal matrix (of size $m$). Observe that $ J^2_n = -I_n.$ Also, since $ s \leq r$ and $ n =2m =r+s$ we get that $s\leq m$. It is easy to see that $((t\pm \pi)\Lambda_m)^\bot=(t\pm \pi)\Lambda_m$.	
From that, we have a  perfect pairing between $(t+ \pi)\Lambda_m$ and $\Lambda_m/(t+ \pi)\Lambda_m$. Note that the last lattice is isomorphic to $(t-\pi)\Lambda_m$ via the multiplication by $t-\pi$. Thus there is an induced pairing 
\[
\{\ , \ \}: (t+\pi)\Lambda_m\times (t-\pi)\Lambda_m\rightarrow O_{F_0}
\]
given by the formula $\{(t+\pi)x,(t-\pi)y\}=((t+\pi)x,y)$. We call $\{\ , \ \}$ the {\it modified pairing} and denote by $W^{\bot'}$ the dual space of $W\subset (t\pm \pi)\Lambda_m$. Note that $W^{\bot'} \subset~(t \mp~\pi) \Lambda_m$, and one has the formula $W^{\bot'} = (t \mp \pi) W^\bot$. \\
By using the modified pair, we can now measure the ``isotropic part" in $\calG_m$. Note that $\calG_m$ is included in $(t + \pi) \Lambda_m$, and is therefore totally isotropic for the original pairing.

\begin{Definition}\label{G2}
For any point $(\calF_m,\calG_m)\in \Mspl$, we define $\calG^{\bot'}_m\subset (t-\pi)\Lambda_m$ the dual space of $\calG_m$ under the modified pairing.
\end{Definition}

\begin{Lemma}\label{lm38}
We have $\calG^{\bot'}_m=((t+\pi)^{-1}\calG_m)^\bot=(t-\pi) \calG_m^\bot$	
\end{Lemma}

\begin{proof}
From the definition of the modified pairing, one gets $\calG^{\bot'}_m = (t-\pi) \calG_m^\bot$ since $\calG_m \subset (t+ \pi) \Lambda_m$. Now, the same relation applied to $\calG^{\bot'}_m$ instead of $\calG_m$ gives
$$\calG_m = (t + \pi) (\calG^{\bot'}_m)^\bot$$
which implies that $\calG^{\bot'}_m = ((t + \pi)^{-1} \calG_m)^\bot$.
\end{proof}

\begin{Proposition}\label{prop39}
 The sublattice $\calG_m^{\bot'}$ is locally free of rank $r$, and we have $\calG_m^{\bot'}\subset \calF_m$ such that
 \[
 (t-\pi)\calF_m\subset \calG_m^{\bot'}, \quad (t+\pi)\calG_m^{\bot'}=0
 \]	
\end{Proposition}

\begin{proof}
Since $\calG_m$ is locally free of rank $s$, and since the modified pairing is perfect between $(t+\pi) \Lambda_m$ and $(t-\pi)\Lambda_m$, one gets that $\calG_m^{\bot'}$ is locally free of rank $n-s = r$. The inclusion $\calG_m \subset \calF_m$ gives by duality $\calF_m \subset \calG_m^\bot = (t - \pi)^{-1} \calG_m^{\bot'}$, so that $(t-\pi)\calF_m\subset \calG_m^{\bot'}$.

The relation $(t + \pi) \calF_m \subset \calG_m$ implies $\calF_m \subset (t+\pi)^{-1} \calG_m$, and taking the orthogonal one gets the inclusion $\calG_m^{\bot'} \subset \calF_m$. Finally, since $\calG_m^{\bot'}$ is inside $(t - \pi) \Lambda_m$, one must have $(t + \pi) \calG_m^{\bot'} = 0$.
\end{proof}

In the rest of this section, we consider the special fiber of $\Mspl$. For any point $x=(\calF_m,\calG_m)\in \Mspl\otimes k$, note that $\calG_m$ (resp. $\calG_m^{\bot'}$) are in the lattice $t\Lambda_m$ with dimension $s$ (resp. $r$).

\begin{Definition}\label{defhl}
Let $x=(\calF_m,\calG_m)\in \Mspl\otimes k$. Set the integers $(h,l)$ as the dimension of $t\calF_m$ and $\calG_m\cap \calG_m^{\bot'}$ respectively.
\end{Definition}

\begin{Lemma}
We have $0\leq h\leq l\leq s$ and $l = s \mod 2$.
\end{Lemma}

\begin{proof}
In special fiber, the space $t \calF_m$ is inside both $\calG_m$ and $\calG_m^{\bot'}$. One has therefore the inclusions
$$t \calF_m \subset \calG_m \cap \calG_m^{\bot'} \subset \calG_m.$$
Taking the dimensions, one gets the inequalities $0 \leq h \leq l \leq s$. 

Now the modified pairing is between $t \Lambda_m$ and itself, and its matrix in the standard basis is given by $- J_n$. It is therefore alternating. If we denote by $f_1, \dots, f_s$ a basis for $\calG_m$, and if $A$ is the matrix $(\{f_i, f_j\})_{i,j}$ then $A$ is an skew-symmetric matrix. Moreover $l$ is equal to the nullity of $A$. Since the rank of a skew-symmetric matrix is even, one gets the relation $l = s \mod 2$.
\end{proof}

The integers $(h,l)$ will give a  stratification on $\Mspl\otimes k$. Note that the spin condition, (see condition (4') of Definition \ref{DefSplSpin}), requires that $h=\dim(t\calF_m)$ has the same parity as $s$. More precisely, we get:
\begin{equation}
\Mspl\otimes k=\coprod_{\substack{0\leq h\leq l\leq s \\ h,l \, \equiv \, s \, \text{mod} \, 2}} X_{h,l}.
\end{equation}
%when $s$ is even, and 
%\begin{equation}
%\Mspl\otimes k=\coprod_{\substack{0\leq h\leq l\leq s \\ h,l={\text{odd}}}} X_{h,l} \end{equation}
%when $s$ is odd.

\section{Closure relations}\label{Closure}
The goal of this section is to compute the closure relations for the strata $X_{h,l}$ of $\Mspl\otimes k$. %over the special fiber. %We use similar arguments as in \cite[Proposition 2.11]{BH}. 
We have 
\[
(0) \subseteq t\calF_m \subseteq \calG_m\cap \calG_m^{\bot'} \subseteq \calG_m \subset t\Lambda_m
\]
with the corresponding ranks $ 0 \leq h \leq l \leq s \leq n$. Also, recall from the previous section that $h$ and $l$ have the same parity as $s$.

\begin{Lemma}\label{1stInclusion}
Let $(h,l)$ be integers with $ 0\leq h\leq l\leq s$ and $h,l = s \mod 2$. Then
\[
\overline{X_{h,l}} \subseteq \coprod_{\substack{0\leq  h' \leq h\leq l \leq l' \leq s \\ h',l' \, \equiv \, s \, \text{mod} \, 2 }} X_{h',l'}
\]
\end{Lemma}

\begin{proof}
As in \cite[Proposition 2.9]{BH}, by the definition of $h$ and $l$, we can deduce that $h$ decreases by specialization while $l$ increases by specialization.
\end{proof}

Now, we are ready to prove the main result of this subsection:
\begin{Proposition}\label{prop40}
Let $(h,l)$ be integers with $ 0\leq h\leq l\leq s$ and $h,l = s \mod 2$. Then
\[
\overline{X_{h,l}} = \coprod_{\substack{0\leq  h' \leq h\leq l \leq l' \leq s \\ h',l' \, \equiv \, s \, \text{mod} \, 2 }} X_{h',l'}.
\]
\end{Proposition}
\begin{proof}
%We use the same strategy as in the proof of \cite[Proposition 2.11]{BH}. 
By Lemma \ref{1stInclusion}, it's enough to show  
\[\coprod_{\substack{0\leq  h' \leq h\leq l \leq l' \leq s \\ h',l' \, \equiv \, s \, \text{mod} \, 2 }} X_{h',l'} \subseteq \overline{X_{h,l}} .\] 
We will use similar arguments as in the proof of \cite[Proposition 2.11]{BH}. Let $\varepsilon$ be the unique element in $\{ 0, 1 \}$ equal to $s$ modulo $2$. We first prove that $X_{\varepsilon,s}$ lies in the closure of $ X_{h,l}$ for any $h,l$ with $ 0\leq h\leq l\leq s$ and $h=l=s \mod 2$. Let $x=(\calF_m,\calG_m)\in X_{\varepsilon,s} (k)$. Given a pair $(h,l)$ with $h \leq l$ and $h=l=s \mod 2$, we will find a generalization of $x$ which lies in $X_{h,l}$. To do this, we construct a lift $\Tilde{x}$ over $k\lp u\rp$ such that the generic fiber lies in $X_{h,l}$. Since $(\calF_m,\calG_m) \in X_{\varepsilon,s}$ we have $ \calG_m \subseteq \calG_m^{\bot'} \subset \calF_m .$ Suppose that  
\[
\calG_m = \text{span}_k \{t f_1, \dots,t f_s\}, \quad \calG_m^{\bot'}  = \text{span}_k \{t f_1, \dots,t f_s,t f_{s+1} \dots, t f_r \}
\]
and 
\[
\calF_m = \text{span}_k \{t f_1, \dots,t f_r,t f_{r+1} \dots, t f_n\}
\]
if $s$ is even; 
\[
\calF_m = \text{span}_k \{t f_1, \dots,t f_r,t f_{r+2} \dots, t f_{n}, f_1\}
\]
if $s$ is odd. By appropriate change of basis we can assume that the corresponding matrix of $\{\, , \, \}$ with respect to the basis $t f_1, \dots, t f_n$ is given by: 
\[
T = \left[\ 
\begin{matrix}[cccc]
&&& I_{s} \\ 
&& I_{\frac{r-s}{2}}&\\
& -I_{\frac{r-s}{2}}&&\\
 -I_{s}&&& 
\end{matrix}\ \right].
\]

Next, we consider a lift of $\calG_m$ given by 
\[
\widetilde{\calG_m} = M_1 =\left [\ 
\begin{matrix}[c]
 I_{s} \\ 
0\\
0\\
 X 
\end{matrix}\ \right],
\]
where $X \in \text{Mat} (uk\lp u\rp)$ and the corresponding sizes are $ (\frac{r-s}{2})\times s$, $ (\frac{r-s}{2})\times s$ and $ s \times s$. If $\varepsilon =1$, we ask that
$$ X = \left [\ 
\begin{matrix}[cc]
0 & 0 \\
0 & X_0
\end{matrix}\ \right],
$$
where $X_0$ is a square matrix of size $s-1$. The orthogonal $\widetilde{\calG_m^{\bot'}}$ of $\widetilde{\calG_m}$ will be a lift for $ \calG_m^{\bot'}$ and is given by 
\[
\widetilde{\calG_m^{\bot'}} = M_2 =\left[\ 
\begin{matrix}
 I_{s} & & \\ 
& I_{\frac{r-s}{2}} & \\
 & & I_{\frac{r-s}{2}} \\
X^t & 0 & 0 
\end{matrix}\ \right].
\]

Finally, we consider a lift $\widetilde{\calF_m}$ of $\calF_m$ inside $\Lambda_m$. Inside this space, we have: 
\[
\widetilde{\calG_m}=  \left[\ 
\begin{matrix}[c]
0_{n\times n}\\ \hline
M_1  
\end{matrix}\ \right], \quad  \widetilde{\calG_m^{\bot'}} = \left[\ 
\begin{matrix}[c]
0_{n\times n}\\ \hline
M_2  
\end{matrix}\ \right], \quad \widetilde{\calF_m} = \left[\ 
\begin{matrix}[c|c]
0_{n\times r} & *_{n \times s}  \\ \hline
M_3 & M_4 
\end{matrix}\ \right]
\]
where 
\[
M_3 = \left[\ 
\begin{matrix}
I_{s}& 0 & 0 \\
0 & I_{\frac{r-s}{2}} & 0  \\ 
0& 0 & I_{\frac{r-s}{2}}\\
X& 0 & 0 
\end{matrix}\ \right]. %, \quad M_4 = \left [\ 
%\begin{matrix}[c]
%0\\ 
%0\\
% 0\\
%  I_{s}  
%\end{matrix}\ \right].
\]
%(We omitted the sizes of the zero matrices that appear in $M_3$ and $M_4$, as they can be easily read from the block decompositions of the matrices.) 
If $\varepsilon = 0$, we set
\[
M_4 = \left [\ 
\begin{matrix}[c]
0\\ 
0\\
 0\\
  I_{s}  
\end{matrix}\ \right], \quad *_{n \times s}  = \left [\ 
\begin{matrix}[c]
W\\ 
0\\
0\\
XW
\end{matrix}\ \right],
\]
where $W$ is a square matrix of size $s$. If $\varepsilon = 1$, we set
\[
M_4 = \left [\ 
\begin{matrix}[cc]
0 & 0\\ 
0 & 0\\
0 & 0\\
0 & 0 \\
0 & I_{s-1}  
\end{matrix}\ \right], \quad *_{n \times s}  = \left [\ 
\begin{matrix}[cc]
e & W\\ 
0 & 0\\
0 & 0\\
0 & XW
\end{matrix}\ \right], \quad e = \left [\ 
\begin{matrix}[c]
1\\ 
0\\
 \vdots\\
0  
\end{matrix}\ \right], \quad W = \left [\ 
\begin{matrix}[c]
0\\ 
W_0      
\end{matrix}\ \right]
\]
where $W_0$ is a square matrix of size $s-1$. These lifts need to satisfy the conditions (1)-(4) of Definition \ref{DefSplSpin}. %: $ (0) \subset \widetilde{\calG_m} \subset\widetilde{\calF_m} \subset \Lambda_m $, $ ( \widetilde{\calF_m}, \widetilde{\calF_m} ) = 0 $, $ t\widetilde{\calF_m} \subset \widetilde{\calG_m} $, $ t \widetilde{\calG_m} = (0)$ and $ \text{rk}(\widetilde{t\calF_m}) = \text{even} $. 
From the above constructions, it's obvious that $(0) \subset \widetilde{\calG_m} \subset\widetilde{\calF_m} \subset \Lambda_m  $, $t \widetilde{\calG_m} = (0)$ and  $ t\widetilde{\calF_m} \subset \widetilde{\calG_m}$. 

It is enough to check the isotropic condition $( \widetilde{\calF_m}, \widetilde{\calF_m} ) = 0 $ from above. Note that $ ( \, , \, ) $ is symmetric and is given by the matrix 
\[
\left[\ 
\begin{matrix}[c|c]
0_{n\times n} & -T  \\ \hline
T & 0_{n\times n} 
\end{matrix}\ \right].
\]
From the isotropic condition, we get
\begin{equation}\label{isot.cond}
W=-W^t, \quad (X-X^t)W = 0
\end{equation}
when $\varepsilon = 0$, and 
\begin{equation}\label{isot.cond1}
W_0=-W_0^t, \quad (X_0 - X_0^t)W_0 = 0
\end{equation}
when $\varepsilon = 1$. Note that the spin condition, i.e. $\text{rk}(\widetilde{t\calF_m}) = s \mod 2$, is automatically satisfied. Indeed, if $s$ is even, we get that $\text{rk}(\widetilde{t\calF_m}) = \text{rk}(W) $ and $W$ is a skew-symmetric matrix. If $s$ is odd, then $\text{rk}(\widetilde{t\calF_m}) = 1 + \text{rk}(W_0) $, which is odd since $W_0$ is a skew-symmetric matrix. 

Finally, observe that when $s$ is even, $ h = \text{rk} (W)$ and $ l = \text{rk}( \text{Ker}(X - X^t))$. For any couple $(h,l)$ with $h,l=s \mod 2$ and $ 0\leq h\leq l\leq s$, one can find matrices $X,W$ such that $ h = \text{rk} (W)$ and $ l = \text{rk}( \text{Ker}(X - X^t))$; this gives a generization $\Tilde{x}$ of $x$ and hence the result. The same reasoning applies when $s$ is odd: in that case one has $h = 1 + \text{rk} (W_0)$ and $ l = 1 + \text{rk}( \text{Ker}(X_0 - X_0^t))$. 

The general case can be deduced by a similar discussion: fix a pair $(h,l)$ with $h=l=s \mod 2$ and $ x \in X_{h,l}$. We construct lifts $\Tilde{x} \in X_{h',l'}\otimes_k k\lp u\rp$ for any pair $(h',l')$ with $h' =l'=s \mod 2$ and $0\leq  h \leq h'\leq l' \leq l \leq s. $ Observe that 
\[
(0) \subseteq t\calF_m \subseteq \calG_m\cap \calG_m^{\bot'} \subseteq \calG_m \subset  (t\calF_m)^{\bot'} \subset  t\Lambda_m
\]
and
\[
(0) \subseteq t\calF_m \subseteq \calG_m\cap \calG_m^{\bot'} \subseteq \calG_m^{\bot'} \subset  (t\calF_m)^{\bot'} \subset  t\Lambda_m.
\]
Suppose that 
\[
t\calF_m = \text{span}_k \{t f_1, \dots,t f_h\}, \,\, \calG_m\cap \calG_m^{\bot'}   = \text{span}_k \{t f_1, \dots,t f_h,t f_{h+1} ,\dots, t f_l \},
\] 
\[
\calG_m = \text{span}_k \{ t f_1, \dots,t f_l,t f_{l+1}, \dots, t f_s \}, \, \calG_m^{\bot'}   = \text{span}_k \{t f_1, \dots,t f_l,t f_{s+1}, \dots, t f_{n-l} \},
\]
and 
\[
(t\calF_m)^{\bot'} = \text{span}_k \{t f_1, \dots,t f_{n-l},t f_{n-l+1}, \dots, t f_{n-h}\}, \,\,t\Lambda_m = \text{span}_k \{t f_1, \dots,t f_n\}.
\]
The corresponding matrix $T$ of $\{\, , \, \}$ is given by:
\[
T = \left[\ 
\begin{matrix}[cccccc]
&&&&& I_{h} \\ 
&&&& I_{l-h}&\\
&& A&&&\\
&&& B&& \\
& -I_{l-h}&&&&\\
-I_{h}&&&&&
 
\end{matrix}\ \right], 
\]
where 
\[
A =\left[\begin{matrix}
 &  I_{\frac{s-l}{2}} \\ 
- I_{\frac{s-l}{2}} &  
\end{matrix}\ \right], \quad B =\left[\begin{matrix}
 &  I_{\frac{r-l}{2}} \\ 
- I_{\frac{r-l}{2}} &  
\end{matrix}\ \right].
\]
Observe that $T$ is skew-symmetric and the matrix of the symmetric form $(\, , \, ) $ is given by 
\[
\left[\begin{matrix}
 &  -T \\ 
T &  
\end{matrix}\ \right].
\]
A lift of $\calG_m$ is given by $\widetilde{\calG_m} = \left[\ \begin{matrix}
\mathbf{a}_1 &\dots  & \mathbf{a}_4
\end{matrix}\ \right]$ where the matrices $\mathbf{a}_1, \dots ,\mathbf{a}_4$ have the corresponding sizes $ h\times n,\, (l-h)\times n,\,  (\frac{s-l}{2})\times n, \, (\frac{s-l}{2})\times n$. In particular, 
\[
\mathbf{a}_1^t =  \left[\ \begin{matrix}
I_{h} & 0_{h\times (n-h)}
\end{matrix}\ \right], \,\, \mathbf{a}_2^t =  \left[\ \begin{matrix}
0_{(l-h)\times h} & I_{l-h} & 0_{(l-h)\times (n-2l)} & Y^t_2 & 0_{(l-h)\times h}
\end{matrix}\ \right],
\]
where $ Y_2\in \text{Mat} (uk\lp u\rp)$ of size $ (l-h)\times (l-h)$ and 
\[
\mathbf{a}_3^t =  \left[\ \begin{matrix}
0_{\frac{s-l}{2} \times l} & I_{\frac{s-l}{2}} & 0_{\frac{s-l}{2}\times (r - \frac{l}{2})}
\end{matrix}\ \right], \,\, \mathbf{a}_4^t =  \left[\ \begin{matrix}
0_{\frac{s-l}{2} \times \frac{s+l}{2}} & I_{\frac{s-l}{2}} & 0_{\frac{s-l}{2} \times r} 
\end{matrix}\ \right].
\]
Next consider the lift $\widetilde{\calF_m} =  \left[\ \begin{matrix}
\mathbf{b}_1 & \mathbf{b}_2 &  \mathbf{b}_3
\end{matrix}\ \right]$ where the matrices $\mathbf{b}_1, \mathbf{b}_2 ,\mathbf{b}_3$ have the corresponding sizes $ 2n \times (n-l),\, 2n \times h ,\,  2n \times l-h $. In particular we define these matrices as follows:
\[
\mathbf{b}_1^t =   \left[\ \begin{matrix}
0_{(n-l) \times n} & I_{(n-l)\times (n-l) } & \mathbf{ c } & 0_{(n-l) \times h}
\end{matrix}\ \right],
\]
where $ \mathbf{ c }^t  =  \left[\ \begin{matrix}
0_{(l-h) \times h} & Y_2 & 0_{ (l-h)\times (n-2l)} \end{matrix}\ \right] $ and 
$\mathbf{b}_2^t =   \left[\ \begin{matrix}
I_{h} & 0_{ h \times (2n-h)} 
\end{matrix}\ \right] $. Lastly,
\[
\mathbf{b}_3^t =   \left[\ \begin{matrix}
0_{(l-h) \times h} & Z^t &  & 0_{(l-h) \times (n-2l) } & (Y_2Z)^t & 0_{(l-h) \times (n+h-l)} & I_{l-h} & 0_{(l-h) \times h}
\end{matrix}\ \right]
\]
where $ Z \in \text{Mat} (uk\lp u\rp)$ of size $ (l-h)\times (l-h)$. We also ask for the pair $ (\widetilde{\calG_m}, \widetilde{\calF_m})$ to satisfy the conditions (1)-(4) of Definition \ref{DefSplSpin}. The conditions (1) and (3) are automatically satisfied from the above constructions. From condition (2) we get
$ Z = -Z^t$ and $ (Y_2 - Y^t_2)Z = 0.$ Also, $ \text{rk}(t \widetilde{\calF_m}) = h + \text{rk} (Z) $ and $\text{rk}( \widetilde{\calG_m} \cap \widetilde{\calG_m^{\bot'} }) = h + \text{rk} ( \text{Ker}(Y_2 - Y^t_2 ))$. Therefore, if we have $ h \leq h' = h+ a \leq l'=h+b \leq l$ for $ a,b \geq 0$ even with $ l-h \geq b \geq a$ then we can find a skew-symmetric matrix $Z$ of rank $a$ and a $Y_2$ such that $\text{rk} ( \text{Ker}(Y_2 - Y^t_2 ) )= b \geq a$ where $ (Y_2 - Y^t_2)Z = 0$ (note that $Y_2 - Y^t_2$ is a skew-symmetric matrix of size $l-h$, which is even).
\end{proof}

Next we show that
\begin{Proposition}\label{irredcomp}
    For all pairs $(h,l)$ with $h=l=s \mod 2$ and $ 0 \leq h\leq l\leq s$, the stratum $X_{h,l}$ is nonempty and smooth, and is
equidimensional of dimension $ rs - \frac{(l-h)(l-h-1)}{2}$. 
\end{Proposition}
\begin{proof}
The proof follows the same reasoning as the proof of \cite[Proposition 2.12]{BH}. The only difference is that the modified pairing is alternating rather than symmetric, which leads to different dimensions for the strata $X_{h,l}$. 
\end{proof}

\begin{Corollary}
The irreducible components of the scheme $\Mspl\otimes k$ are the closures of $X_{h,h}$ for $h = s \mod 2$ and $0 \leq h \leq s$. 
\end{Corollary}
\begin{proof}
It follows from Proposition \ref{irredcomp} and the fact that $ l \neq h+1$ since $l,h$ have the same parity.
\end{proof}
\begin{Proposition}
The smooth locus of $\Mspl\otimes k$ is the union of the strata $X_{h,h}$ for $h = s \mod 2$ and $0 \leq h \leq s$.
\end{Proposition}
\begin{proof}
From the above it is enough to prove that for each pair $(h,l)$ with $h=l=s \mod 2$ with $h < l $ there is $x \in X_{h,l}$ such that $\Mspl\otimes k$ is not smooth at $x$. The proof is similar to the proof of \cite[Proposition 2.14]{BH}.

We use the notation/basis of the general case in the proof of Proposition \ref{prop40}. In addition, we let $ \calF_m = \text{span}_k \{f_1, \dots,f_h,t f_1, \dots,t f_{n-h} \} $. Next we lift $ x = (\calG_m, \calF_m)$ to $x' = (\calG_m', \calF_m' ) \in \Mspl\otimes_k k[\epsilon]/\epsilon^2$ such that 
\[
\calG_m' =\text{span}_k \{t f_1, \dots,t f_{l-1} - \epsilon t f_{n-h}, t f_l+\epsilon t f_{n-h-1}, \dots , t f_s \} 
\]
and 
\begin{align*}
\calF_m' &= \text{span}_k \{ f_1, \dots, f_{h}, t f_1, \dots, t f_{l-1}-\epsilon t f_{n-h}, \\
&\qquad t f_l+ \epsilon t f_{n-h-1}, \dots, t f_{n-h-1} + \epsilon f_l, t f_{n-h} - \epsilon f_{l-1} \}. 
\end{align*}
This pair needs to satisfy the conditions (1)-(4) of Definition \ref{DefSplSpin}. We can easily see that all the conditions except (2) are automatically satisfied. So we need to check $(\calF_m' ,\calF_m' ) =0$. First observe that
$(t f_i , f_{n-h+i}) = 1  $ for $1\leq i \leq h$ and $(t f_{h+i} , f_{n-l+i}) = 1  $ for $1\leq i \leq l-h$. Also $(t f_{l+i} , f_{\frac{s+l}{2}+i}) = 1  $ for $1\leq i \leq \frac{s-l}{2}$ and $(t f_{s+i} , f_{s+\frac{r-l}{2}+i}) = 1  $ for $1\leq i \leq \frac{r-l}{2}$. Hence, $(\calF_m' ,\calF_m' ) =0$ since  
\[
(t f_{n-h-1}+ \epsilon f_l , t f_{n-h} - \epsilon f_{l-1}) = \epsilon - \epsilon = 0, \,\, (t f_{n-h-1}+ \epsilon f_l , t f_{l-1} - \epsilon t f_{n-h}) = \epsilon^2= 0.
\]
 We will now prove that this point cannot be lifted to $ k[\epsilon]/\epsilon^3$. Suppose it can and then consider the elements $u_1 = t f_{l-1} - \epsilon t f_{n-h} + \epsilon^2 t u \in \calG''_m$ and $ v_1 =  t f_{n-h-1}+\epsilon f_{l} + \epsilon^2 f_{n-h-1}+ \epsilon^2 v \in \calF''_m $ such that $ \epsilon^2 t v \in  \calG''_m$. Then $ (u_1, v_1 ) = \epsilon^2+ \epsilon^2  = 2 \epsilon^2 \neq 0$ and so, $ ( \calF''_m , \calF''_m ) \neq 0$. 
\end{proof}

\begin{Remark}{\rm
Even though the irreducible components are the closures of $X_{h,h}$, and $X_{h,h}$ is smooth,  it is not true in general that the irreducible components themselves are smooth. In particular, our computations show that the irreducible components that correspond to $X_{h,h}$ with $ 1 <h<s$ are not smooth. Therefore, ${\rm M}^{\rm spl}$ is not semi-stable when $s > 3$.
}\end{Remark}
% Also, our computations show that the irreducible components that correspond to $ 1 <h<s$ are not smooth. Thus, ${\rm M}^{\rm spl}$ is not semi-stable when $s > 3$. 

\section{Flatness and  Reducedness of $\Mspl$}\label{Flatness}
Our goal in this section is to show that the splitting model $\Mspl$ is flat, normal and Cohen-Macaulay over $\text{Spec}\, O_F$ and its special fiber is reduced. We will first state and prove the theorem, assuming the computation of the local rings which will be done in the second part.
 
\subsection{The main result } 

Recall that $\varepsilon$ is equal to $0$ if $s$ is even, and $1$ otherwise. We set $s'= s - \varepsilon$ so that $s'$ is an even integer.

\begin{Theorem}\label{flatEvenCase}
For any signature $(n-s,s)$, the splitting model $\Mspl$ is flat, normal over $O_F$ and its special fiber $\Mspl\otimes k$ is reduced. 
\end{Theorem}
\begin{proof}
As in the proof of \cite[Proposition 2.16]{BH}, it's enough to consider a local ring at $x \in X_{\varepsilon,s}$. The calculations detailed in the next section give that such a local ring is given by
\[
W_{\mathcal{O}}(k) [U_1, \dots, U_a,X,Y,Z, W ] / (W+W^t, \,\, (Z-Z^t+X^tY-Y^tX)W -2\pi I_{s'}),
\]
where $W_{\mathcal{O}}(k) = W(k) \otimes_{\mathbb{Z}_p} O_F  $ and the sizes of $ X,Y,Z, W $ are $\frac{r-s}{2}\times s', \frac{r-s}{2}\times s', s' \times s', s'\times s'$ respectively, and $a$ is an integer depending on $s$. Observe that
\[
\scalebox{1.1}{$
\frac{
    W_{\mathcal{O}}(k)[U,X,Y,T,S,W]
}{
    \left( W + W^t, \, (T + X^t Y - Y^t X) W - 2\pi I_{s'} \right)
}
\longrightarrow
\frac{
    W_{\mathcal{O}}(k)[U,X,Y,Z,W]
}{
    \left( W + W^t, \, (Z - Z^t + X^t Y - Y^t X) W - 2\pi I_{s'} \right)
}
$}
\]
given by $ T \mapsto Z - Z^t$ and $S \mapsto Z+Z^t$. Note that $T$ is skew-symmetric and $S$ is symmetric. We obtain the following isomorphism 
\[
\scalebox{1.1}{$
\begin{aligned}
\text{Spec}\frac{
    W_{\mathcal{O}}(k)[U,X,Y,Z,W]
}{
    \left( W + W^t, \, (Z - Z^t + X^t Y - Y^t X) W - 2\pi I_{s'} \right)
}
\cong \\
 \mathbb{A}^{a+\frac{s(s+1)}{2}}_{W_{\mathcal{O}}(k)} \times \text{Spec} \frac{
    W_{\mathcal{O}}(k)[X,Y,T,W]
}{
   \left( W + W^t, \, T+T^t,\, (T + X^t Y - Y^t X) W - 2\pi I_{s'} \right)
}
\end{aligned}
$}
\]
and the last affine scheme in the above isomorphism is isomorphic to 
\[
\text{Spec} \frac{
    W_{\mathcal{O}}(k)[X,Y,T,W]
}{
   \left( W + W^t, \, T+T^t,\, T W - 2\pi I_{s'} \right)
}
\]
given by $X \mapsto X$, $ Y \mapsto Y$ and $ T \mapsto T - (X^tY - Y^t X)$. So, from the above it's enough to study and prove the desired properties for the affine scheme
\begin{equation}
\Spec R = \text{Spec} \frac{
    W_{\mathcal{O}}(k)[X,Y,T,W]
}{
   \left( W + W^t, \, T+T^t,\, T W - 2\pi I_{s'} \right)
}.	
\end{equation}

From Lemma \ref{Reducedness}, we get that the special fiber of $\Spec R$ is reduced. To show that $\Spec R$ is flat, it is enough to show that it's topologically flat. We use a similar argument as in \cite[Remark 5.9]{DePa}. In particular, we have that $GL_n$ acts on $\Spec R$ by $(T,W) \rightarrow (uTu^t, (u^t)^{-1} W u)$. The action of $GL_n(k)$ on $(\Spec R)(k)$ has a finite number of orbits. After a suitable coordinate change by $u$ we get 
\[
T = \begin{pmatrix}
    0 & \begin{pmatrix}
        0 & T_0 \\
        0 & 0
    \end{pmatrix}  \\
    \begin{pmatrix}
        0 & 0 \\
        -T_0^t & 0
    \end{pmatrix} & 0
\end{pmatrix}, \,\, W =\begin{pmatrix}
    0 & \begin{pmatrix}
        0 & 0 \\
        W_0 & 0
    \end{pmatrix}  \\
    \begin{pmatrix}
        0 & -W_0^t \\
        0 & 0
    \end{pmatrix} & 0
\end{pmatrix}
\]
with $T_0, W_0$ non-degenerate. Observe that $T=-T^t$, $ W = - W^t$ and $TW = 0$. For any lifting $ \Tilde{T_0}, \Tilde{W_0}$, we construct 
\[
\scalebox{0.75}{$
\widetilde{T} =  
\begin{pmatrix}[cccc]
&&& \widetilde{T_0} \\ 
&& - 2\pi (\widetilde{W_0}^t)^{-1} &\\
& 2\pi (\widetilde{W_0})^{-1} &&\\
 -\widetilde{T_0}^t&&& 
\end{pmatrix}, \quad \widetilde{W} = 
\begin{pmatrix}[cccc]
&&& - 2\pi (\widetilde{T_0}^t)^{-1} \\ 
&& \widetilde{W_0} &\\
& -\widetilde{W_0}^t &&\\
  2\pi (\widetilde{T_0})^{-1}&&& 
\end{pmatrix}
$}
\]
where $ \widetilde{T} = -  \widetilde{T}^t $, $ \widetilde{W} = -  \widetilde{W}^t $ and $ \widetilde{T} \widetilde{W} = 2\pi I_{s'}.$ Thus, $\Spec R$ is topologically flat and from the above we deduce that $\Spec R$ is flat.

Lastly, normality of $\Spec R$ easily follows since $\Spec R$ is flat with reduced special fiber and smooth generic fiber.
\end{proof}
\begin{Lemma}\label{Reducedness}
 Assume that $B$ is a commutative reduced ring and $2$ is invertible in $B$. Also, assume that $m$ is even and set $B' =  \frac{
   B[T,W]
}{\left( W + W^t, \, T+T^t,\, T W \right)}$ where $W= (w_{i,j})$, $T = (t_{i,j})$ are matrices of size $m \times m$. Then $B'$ is a reduced ring.   
\end{Lemma}
\begin{proof}
The proof uses the same steps that were used in \cite[\S 4]{Kotzev} for the symmetric case. As in loc. cit., the main ingredient of the proof is the existence of a $B$-free basis of the quotient ring $B'$. In our case, such a $B$-basis was formed by the ``standard tableaux pairs $ (C_i,D_j)$" in \cite[Prop. 1.6]{Ba}, such that each pair of tableaux $(C,D)$ can be written as a linear combination of standard tableaux pairs $(C_i,D_j)$ (see \cite[\S 1]{Ba} for the explicit definition of the standard pairs). The rest of the proof will follow by an induction on $m$.

If $m=2$, then it is easy to see that $B'$ is reduced. Assume that $m \geq 4$ and let $ J \subset B'$ be the ideal generated by the maximal Pfaffian of $T$. By Proposition \cite[Cor. 1.3]{Ba}, $J$ does not contain nilpotents and so it's enough to prove that $B_1:=B'/J$ is reduced. Note that $B_1$ has a $B$-basis given by the images of standard pairs $(C_i,D_i)$. From the construction of these pairs, see \cite[\S 1]{Ba}, it follows that the product of $w_{1,2}$ with any standard pair is again a standard pair and two different standard pairs remain different, when both multiplied by $w_{1,2}$. Thus, as in \cite[Lemma 4.2]{Kotzev}, we deduce that $w_{1,2}$ is a non-zerodivisor in $B_1$. 

%By a standard localization argument for pfaffian ideals, 
Next, invert $w_{1,2}$ and mimick \cite{JP} (and \cite{Kotzev}) to get
\begin{equation}\label{reducesize}
B_1 [w_{1,2}^{-1}] \cong A[T',W']/ (W' + (W')^t, \, T'+(T')^t,\, T' W')
\end{equation}
where $T'$, $W'$ are matrices of size $ (m-2) \times (m-2)$ and $A = B[w_{1,2},w^{-1}_{1,2}, w_{1,3}, \dots, w_{1,n},\\  w_{2,3}, \dots, w_{2,n}]$. Consider the injective map $B_1 \rightarrow B_1 [w_{1,2}^{-1}]$. By the inductive hypothesis and (\ref{reducesize}), $B_1 [w_{1,2}^{-1}]$ is reduced and thus $B_1$ is also reduced.   
\end{proof}

\begin{Theorem}
For any signature $(n-s,s)$, the splitting model $\Mspl$ is Cohen-Macaulay over $O_F$. 	
\end{Theorem}
\begin{proof}
From all the above discussion, it's enough to prove that $\Spec R$ is Cohen-Macaulay and since $\Spec R$ is $O_F$-flat, it is enough to prove that its special fiber 
\[
\overline{\Spec R} = \mathbb{A}^{(r-s)s'}\times \text{Spec} \frac{
    W_{\mathcal{O}}(k)[T,W]
}{
   \left( W + W^t, \, T+T^t,\, T W \right)
}
\]
is Cohen-Macaulay. Set $R'=W_{\mathcal{O}}(k)[T,W]/  \left( W + W^t, \, T+T^t,\, T W \right)$. From \cite[\S 2]{Ba}, we obtain the isomorphism
 \[
  R' \simeq S(I)/(a)
   \] 
where the ideal $I$ is generated by the $(s'-2)$-order Paffians of $T$, $S(I)$ is the symmetric algebra of $I$ and $a$ is a non-zero divisor in $S(I)$ (see loc. cit. for the explicit definition of these terms).
%Note that the ring $R'=W_{\mathcal{O}}(k)[T,W]/  \left( W + W^t, \, T+T^t,\, T W \right)$ has a $W_{\mathcal{O}}(k)$-basis formed by the ``standard tableaux pairs", such that each pair of tableaux can be written as a linear combination of standard tableaux pairs. We refer to \cite[Prop. 1.6]{Ba}, \cite{BoNe} for details.  By using the $W_{\mathcal{O}}(k)$-basis, we obtain the symmetric algebra $S(I)$ satisfying \[S(I)/(a)\simeq R',\] the ideal $I$ is generated by the $(s'-2)$-order Paffians of $T$, and $a$ is a non-zero divisor in $S(I)$ (see \cite[\S 2]{Ba}). 
By \cite[Prop. 2.1]{Ba} the symmetric algebra $S(I)$ is isomorphic to the Rees algebra $\calR(I)$ and by  \cite{Ba2} (see also \cite{BR}) $\calR(I)$ is Cohen-Macaulay. Thus, $S(I)$ is Cohen-Macaulay and so is $R'$ since $a$ is a non-zero divisor.
\end{proof}

\subsection{Local ring calculations}

In this section, we will compute the local ring at a point located in the closed stratum $X_{\varepsilon, s}$.

\begin{Proposition}
Let $x$ be a point in $X_{\varepsilon, s}$. Then the local ring at $x$ is isomorphic to the localization at $W=Z =\pi= 0$ of
\[
W_{\mathcal{O}}(k) [U_1, \dots, U_a,X,Y,Z, W ] / (W+W^t, \,\, (Z-Z^t+X^tY-Y^tX)W -2\pi I_{s'}),
\]
where $W(k)$ is the Witt ring of $k$ and $W_{\mathcal{O}}(k) = W(k) \otimes_{\mathbb{Z}_p} O_F  $. The sizes of $ X,Y,Z, W $ are $\frac{r-s}{2}\times s', \frac{r-s}{2}\times s', s' \times s', s'\times s'$ respectively, and $a$ is an integer depending on $s$.
\end{Proposition}

\begin{proof}
We will give the notations as in the previous section. A lift for $\calG_m$ is given by
$$\widetilde{\calG_m} = \left [\ 
\begin{matrix}[c]
\pi M \\
M 
\end{matrix}\ \right], \quad\text{where}~ M = \left [\ 
\begin{matrix}[c]
I_s\\ 
X\\
Y\\
Z  
\end{matrix}\ \right].$$
The size of the matrices $X,Y$ is $\frac{r-s}{2} \times s$, and $Z$ is a square matrix of size $s$. The space $\widetilde{\calF}_m$ will automatically contain 
$$ \left [\ 
\begin{matrix}[c]
-\pi N \\
N 
\end{matrix}\ \right], \quad\text{where}~ N = \left [\ 
\begin{matrix}[cc]
0 & 0\\ 
I_{\frac{r-s}{2}} & 0 \\
0 & I_{\frac{r-s}{2}}\\
Y^t & - X^t  
\end{matrix}\ \right].$$
Finally, one adds to $\widetilde{\calF}_m$
$$\left [\ 
\begin{matrix}[c]
MB - \pi N_1 \\
N_1 
\end{matrix}\ \right], \quad\text{where}~ N_1 =
\left [\ 
\begin{matrix}[c]
0\\ 
0\\
0\\
A  
\end{matrix}\ \right].$$
The matrices $A, B$ are square matrices of size $s$. When $s$ is even, one has $A = I_s$, and $B$ is a matrix with indeterminate coefficients. When $s$ is odd, one has
\[
A = \left [\ 
\begin{matrix}[cc]
a_0 & A_1 \\
0 & I_{s-1}
\end{matrix}\ \right], \qquad B = \left [\ 
\begin{matrix}[cc]
1 & 0 \\
B_0 & B_1
\end{matrix}\ \right].
\]
where $B_1$ is a square matrix of size $s-1$. \\
The only condition to check for $\widetilde{\calF}_m$ is the isotropy condition, from which we have
$$A^t B + B^t A = 0, \qquad  (Z-Z^t+X^tY-Y^tX)B = 2 \pi A. $$
One gets the desired result when $s$ is even by setting $W = B$ and $a=0$. \\
When $s$ is odd, one gets
$$a_0 = 0, \quad A_1 = - B_0^t, \quad B_1 + B_1^t = 0.$$
Let $Q$ be the matrix $Z-Z^t+X^tY-Y^tX$, and write 
$$Q = \left [\ 
\begin{matrix}[cc]
0 & -Q_0^t \\
Q_0 & Q_1
\end{matrix}\ \right]$$
where $Q_1$ is a square matrix of size $s-1$, and $Q_0$ is a vector of size $s-1$. \\
The remaining equations are $Q_1 B_1 = 2\pi I_{s-1}$ and  $Q_0 = - Q_1 B_0$. Let us decompose the matrices
$$Z = \left [\ 
\begin{matrix}[cc]
z & Z_2 \\
Z_1 & Z_3
\end{matrix}\ \right] \qquad X = \left [\ 
\begin{matrix}[cc]
X_1 & X_2 
\end{matrix}\ \right] \qquad Y = \left [\ 
\begin{matrix}[cc]
Y_1 & Y_2
\end{matrix}\ \right] $$
where $Z_3$ is a square matrix of size $s-1$, $X_1, Y_1$ are vectors of size $\frac{r-s}{2}$ and $X_2, Y_2$ are $\frac{r-s}{2} \times (s-1)$ matrices. Then the first equation is 
$$(Z_3 - Z_3^t + X_2^t Y_2 - Y_2^t X_2) B_1 = 2 \pi I_{s-1}$$
which is of the desired form (setting $W=B_1$). The second equation gives the value of $Z_1$ in terms of $Z_2, X, Y, B_0$. One gets the result, with $a=r+s-1$, the extra variables corresponding to $z, Z_2, X_1, Y_1, B_0$.
\end{proof}

 \begin{Remark}{\rm
The spin condition is automatically satisfied. Indeed, the integer $h$ is given by the rank of the matrix $B$. It is thus equal to $\varepsilon + rk W$, and since $W$ is skew-symmetric, one has $h = s \mod 2$. 

The spin condition appears naturally: indeed if $s$ is odd, the same computation holds for the local rings at points in $X_{0,s}$. However, the equation $(Z-Z^t+X^tY-Y^tX)W = 2 \pi I_s$ with $W$ skew-symmetric cannot be verified in characteristic $0$ since $W$ cannot be invertible.
}\end{Remark}

\section{Applications}\label{appl}

\subsection{Relation of $X_{h,l}$ with Schubert cells} In this section, we will study the relationship between the stratification of the splitting model $\Mspl_{I}$ over the special fiber ($\pi=0$), described by the strata $X_{h,l}$ (see Definition \ref{defhl}), and the stratification of the local model $\Mloc_{I}$ over the special fiber, given by the Schubert cells indexed by the $\mu$-admissible set (see \cite[\S 2]{PR}). 

We focus on two special maximal parahoric subgroups for general signature $(n-s,s)$: the self-dual case and the $\pi$-modular case. The splitting model for the selfdual case, i.e. $I = \{0\}, n=2m+1$, has been studied in \cite{BH}. In this paper, we consider the splitting model for the $\pi$-modular case, i.e. $I= \{m\}$ and $n =2m$. In both cases $\Mspl_{I}$ is flat and over the special fiber we have 
\[
\Mspl_{\{0\}} \otimes k=\coprod_{0\leq h\leq l\leq s} X_{h,l}, \quad \Mspl_{\{m\}}\otimes k=\coprod_{\substack{0\leq h\leq l\leq s \\ h,l \, \equiv \, s \, \text{mod} \, 2}} X_{h,l}.
\]
On the other hand, for the above two level subgroups, the flat local model $ \Mloc_{I}$ has been studied in \cite{PR} and its special fiber decomposes into a disjoint union of certain $\calG_I\otimes k$-equivarient Schubert cells $C_w$ parametrized by the $\mu$-admissible set
\[
\Mloc_{I} \otimes k = \coprod_{w \in \text{Adm}_I(\mu)} C_w.
\]
%(See \cite[\S 2]{PR} and \cite[Theorem 11.3]{PR3} for the explicit definition of the admissible set $\text{Adm}(\mu)$.) 

\begin{Proposition}\label{prop 61}
a) Assume $I = \{0\}$ and pick $C_{w_0} \in \Mloc_{\{0\}} \otimes k$. There exists a unique $ 0 \leq h_0 \leq s$ such that
\[\tau^{-1}(C_{w_0})  =  \coprod_{0\leq h_0\leq l\leq s} X_{h_0,l}.\] 

b) Assume $I = \{m\}$ and pick $C_{w_0} \in \Mloc_{\{m\}} \otimes k$. There exists a unique $ 0 \leq h_0 \leq s$ with $h_0 \, \equiv \, s \, \text{mod} \ 2$ such that
\[\tau^{-1}(C_{w_0})  =  \coprod_{\substack{0\leq h_0\leq l\leq s \\ l \, \equiv \, s \, \text{mod} \, 2}} X_{h_0,l}.\] 
    
\end{Proposition}
\begin{proof}
The proof is similar for both (a) and (b). First observe that the parahoric subgroup $\calG_I\otimes k$ acts transitively on the strata $X_{h,l}$  and on the Schubert cells $ C_w $. Indeed, let $x$ be a point in $X_{h,l} (k)$, which corresponds to spaces $\calF_m, \calG_m$ inside $\Lambda_m$. Let $f_1, \dots, f_h$ be a basis for the complement of $\mathcal{F}_m \cap t \Lambda_m$ inside $\mathcal{F}_m$. This family is isotropic, since $\mathcal{F}_m$ is totally isotropic. This gives a basis $t f_1, \dots, t f_h$ of $t \mathcal{F}_m$, that one completes in a basis $tf_1, \dots, t f_l$ for $\calG_m\cap \calG_m^{\bot'}$. This family is isotropic for the modified pairing. One can complete this family into a basis $tf_1, \dots, t f_n$ for $t \Lambda_m$ such that $\mathcal{G}_m$ is generated by $t f_1, \dots, tf_s$, and the matrix of the modified pairing is
$$ \left [\ 
\begin{matrix}[cccccc]
& & & & & I_h\\ 
& & & & I_{l-h} &\\
 & & I_{s-l} & & & \\
 & & & I_{r-l} & & \\
 & I_{l-h} & & & &  \\
I_h & & & & & 	
\end{matrix}\ \right],$$
in case (a) and
$$ \left [\ 
\begin{matrix}[cccccc]
& & & & & -I_h\\ 
& & & & - I_{l-h} &\\
 & & A_{s-l} & & & \\
 & & & A_{r-l} & & \\
 & I_{l-h} & & & &  \\
I_h & & & & & 	
\end{matrix}\ \right],$$
in case $(b)$, where $A_{2 \kappa} = \left [\ 
\begin{matrix}[cc]
0 & - I_\kappa \\
I_\kappa & 0 \end{matrix}\ \right]$. All is left to do is to find the vectors $f_{h+1}, \dots, f_n$, whose multiplication by $t$ gives $t f_{h+1}, \dots, t f_n$, and such that $(f_i, f_j)=0$ (this is the original pairing). This is done inductively; for example if $f_{h+1}$ is a vector whose multiplication by $t$ give $t f_{h+1}$, let $\alpha_i = (f_i, f_{h+1})$ for $1 \leq i \leq h+1$. If $f_{h+1}' = f_{h+1} + tx$, then the equations $(f_i, f_{h+1}') = 0$ are equivalent to $(f_i, tx) = - \alpha_i$, or $\{tf_i, tx\} = - \alpha_i$. In case (b), one also requires that $(f_{h+1}', f_{h+1}') = 0$, which gives $\alpha_{h+1} + 2 (tx, f_{h+1}) = 0$, or $2 \{tx, tf_{h+1} \} = - \alpha_{h+1}$. In all cases, this is indeed possible. 

Since, $\tau$ is a surjective and $\scrG_I\otimes k$-equivariant morphism we get that there is a pair $(h_0,l) $ such that $ \tau(X_{h_0,l}) = C_{w_0}$. Since $\tau$ is the forgetful morphism which is given by $(\mathcal{F},\mathcal{G}) \mapsto \mathcal{F}$ we deduce that all $\mathcal{F} \in C_{w_0}$ have $\text{rk}(t\mathcal{F} ) = h_0.$ Next, pick an $\mathcal{F} \in C_{w_0}$. Then, as in the proof of Proposition \ref{prop40} (see also the proof \cite[Prop. 2.16]{BH} for the selfdual case), we can always pick a certain basis and construct a $\mathcal{G}$ with $ (\mathcal{F},\mathcal{G}) \in X_{h_0,l}$ for any $ h_0 \leq l \leq s $ modulo the spin condition in case (b). 
\end{proof}

\subsection{Demazure Resolutions} In this section, we will give affine Demazure resolutions of Schubert varieties by considering closed subschemes in the splitting models. Note that Richarz considered similar resolutions in the almost $\pi$-modular case, i.e., $n=2m+1, I=\{m\}$ in \cite{Richarz}. In the presentation of the material we follow \cite{PR}, \cite{Richarz}.

Assume that  $k\llp u\rlp/k\llp t\rlp$ is a quadratic ramified extension with ring of integers $k\lp u\rp/k\lp t\rp$ with $u^2=t$. Set $\Gamma=\Gal(\llp u\rlp/k\llp t\rlp)$. We define the standard lattices $\lambda_i$:
\[
\lambda_i=\text{span}_{k\lp u\rp} \{u^{-1}e_1, \dots, u^{-1}e_i, e_{i+1}, \dots, e_n\}.
\]
For $i\in \ZZ$, we fix isomorphisms compatible with the action of $\pi\otimes 1$ resp. $u\otimes 1$
\[
\lambda_i\otimes_{k\lp t\rp}k\simeq \Lambda_i\otimes_{O_{F_0}}k.
\]
Let $R$ be a $k$-algebra. For any $R$-valued point of the special fiber of $\Mspl$ (resp. $\Mloc$), we have
\begin{equation}\label{eq 601}
0\subset \calG_i\subset \calF_i\subset (\Lambda_i\otimes_{O_{F_0}}k)\otimes_k R=\lambda_i\otimes_{k\lp t\rp}R.	
\end{equation}
Let $\calL_{\calF_i}$ resp. $\calL_{\calG_i}$ be the inverse image of $\calF_i$ resp. $\calG_i$ under the canonical projection
\[
\lambda_i\otimes_{k\lp t\rp} R\lp t\rp\rightarrow \lambda_i\otimes_{k\lp t\rp} R.
\]
By using (\ref{eq 601}), there is a natural embedding from $\Mloc_I\otimes k$ to the corresponding twisted affine flag variety (see \cite{PR3} for details). In particular, when $I=\{0\}$ or $I=\{m\}$, this embedding maps to the twisted affine Grassmannian.
Let $\scrG_{0}$ (resp. $\scrG_{m}$) be the parahoric subgroup with $I=\{0\}$ for $n=2m+1$ (resp. $I=\{m\}$ for $n=2m$). Let
\[
\Gr_{\scrG_0}=LGU_n/L^+\scrG_0,\quad \Gr_{\scrG_m}=LGU_n/L^+\scrG_m
\]
be the corresponding twisted affine Grassmannians. By \cite[Theorem 4.1]{PR3}, there are functional bijection between $\Gr_{\scrG_0}(R)$ (resp. $\Gr_{\scrG_m}(R)$) and sets of lattices:
\[
\begin{array}{l}
\Gr_{\scrG_0}(R)=\{\calL_0\subset R\llp u\rlp^n\mid \calL_0^\vee=\calL_0 \},\\
\Gr_{\scrG_m}(R)=\{\calL_m\subset R\llp u\rlp^n\mid \calL_m^\vee=u\calL_m \}.
\end{array}
\]
We can extend the above lattices to partial selfdual periodic $R\lp u\rp$-lattice chains. More precisely, we have $\cdots\subset u\calL_0\subset \calL_0\subset u^{-1}\calL_0\subset \cdots$ (resp. $\cdots\subset u\calL_m\subset \calL_m\subset u^{-1}\calL_m\subset \cdots$) with $\calL_0\subsetneq u^{-1}\calL_0^\vee=u^{-1}\calL_0$ (resp. $\calL_m= u^{-1}\calL_m^\vee\subsetneq u^{-1}\calL_m$).

As in \cite[\S 1]{PR3}, we denote by $T$ the standard maximal torus of $G$ and by $S$ the maximal split torus of $T$. Using $G(F)\simeq \GL_n(F)\times F^\times$, we identify $X_*(T)$ with $\ZZ^n\times \ZZ$, thus $\mu_{r,s}=(1^{(s)}, 0^{(n-s)},1)$. The image $\lambda_s$ of $\mu_{n-s,s}$ in $X_*(T_{ad})_{\Gamma}=\ZZ^m$ is equal to $(1^{(s)}, 0^{(m-s)})$. The admissible set for the selfdual case is described in \cite[\S 2]{PR}:  
\[
{\rm Adm}_0(\mu_{n-s,s})=\{\lambda_s> \lambda_{s-1}>\cdots \lambda_{1}>\lambda_0=0\}.
\] 
Similarly, for $\pi$-modular case, we have
\[
{\rm Adm}_m(\mu_{n-s,s})=\{\lambda_s> \lambda_{s-2}>\cdots \}.
\]
with the chain ending in $\lambda_1$ if $s$ is odd and in $\lambda_0$ if $s$ is even.

For $i=0,\dots, m$, define
\[
e^{-\lambda_i}:={\rm diag} (u^{(i)}, 1^{(m-i)}, (-1)^i, 1^{(m-i)}, (-u^{-1})^{(i)})
\]
if $n$ is odd, and 
\[
e^{-\lambda_i}:={\rm diag} (u^{(i)}, 1^{(m-i)}, 1^{(m-i)}, (-u^{-1})^{(i)})
\]
if $n$ is even. By \cite[\S 8]{PR3}, we have

\begin{Proposition}
The Schubert variety $S_i\subset \Gr_{\scrG_0}$ (resp. $S_i\subset \Gr_{\scrG_m}$) has dimension $i(n-i)$. It is set theoretically a disjoint union of $L^+\scrG_{0}$ (resp. $L^+\scrG_{m}$)-orbits given by 
\[
S_i=\coprod_{k=0}^i L^+\scrG_0 e^{-\lambda_k} L^+\scrG_0/L^+\scrG_0,
\]
\[
({resp.}\quad S_i=\coprod_{k=s ~{\rm mod}~ 2}^i L^+\scrG_m e^{-\lambda_k} L^+\scrG_m/L^+\scrG_m),
\]
equipped with the reduced scheme structure.
\end{Proposition}

Denote by $C_i=L^+\scrG_0 e^{-\lambda_i} L^+\scrG_0/L^+\scrG_0 $ (resp. $C_i=L^+\scrG_m e^{-\lambda_i} L^+\scrG_m/L^+\scrG_m$) the Schubert cell in $S_i$. 
\begin{Lemma}
There is a natural embedding $\Mloc_{I}\otimes k\hookrightarrow \Gr_{\scrG_I}$ given by $\calF\mapsto u^{-1}\calL_{\calF}$ for $I=\{0\}, n=2m+1$ or $I=\{m\}, n=2m$. The inverse image of $C_i$ is
\[
\{ \calF \in \Mloc_{I}\otimes k\mid rank(t\calF)=i\}.
\]
\end{Lemma}

\begin{proof}
For $I=\{0\}, n=2m+1$, note that the isotropic condition $\bb\calF,\calF\pp=0$ translates to $\calL_\calF=u^2\calL_\calF^\vee$: Since $\bb \ ,\ \pp$ is a nondegenerate alternating form on $u^{-1}\phi: (\lambda_0/u^2\lambda_0)\times (\lambda_0/u^2\lambda_0)\rightarrow k$, we obtain
\begin{flalign*}
\calL_\calF=\calL_{\calF^\bot}&=\{x\in \lambda_0\mid u^{-1}\phi(x,\calL_\calF)\subset u\cdot k\lp u\rp\}\\
&=\{x\in \lambda_0\mid \phi(x,\calL_\calF)\subset u^2\cdot k\lp u\rp\}\\
&=u^2\calL_\calF^\vee.
\end{flalign*}
By setting $\calL_0=u^{-1}\calL_{\calF}$, the above relation translates to $\calL_0=\calL_0^\vee$. Similarly, for $I=\{m\}, n=2m$ we get $\calL_m:=u^{-1}\calL_{\calF}=u^{-1}\calL_m^\vee$ by $\phi: (\lambda_m/u^2\lambda_m)\times (\lambda_m/u^2\lambda_m)\rightarrow k$. Therefore, we have an embdedding $\Mloc_{I}\otimes k\hookrightarrow \Gr_{\scrG_I}$.

By the proof of Proposition \ref{prop 61}, the inverse image $\tau^{-1}(C_i)$ is $\coprod X_{i,l}$ for a fixed $i$, so that $C_i$ corresponds to condition rank($t\calF)=i$. %gives us that $\calL_0$ (resp. $\calL_m$) is a orbit of $e^{-\lambda_i} L^+\scrG_0/L^+\scrG_0$ (resp. $e^{-\lambda_i} L^+\scrG_m/L^+\scrG_m$).
%From the above proposition, we can see that the Schubert cells $ S_h$ that occur in $\Mloc_{I}\otimes k$ can be indexed by $h$ which is the rank of $\pi$ acting on $ \mathcal{F}$ at each point (modulo the spin condition in case (b)).
%is equivalent to rank($u\calL_\calF/u^2\lambda_0)=$rank($\calL_0/\lambda_0)=i$ in the self-dual case, and is equivalent to rank($\calL_m/\lambda_m)=i$ in the $\pi$-modular case, which is in the $L^+\calG_0$ (resp. $L^+\calG_m$)-orbit of $\lambda_i=e^{-\lambda_i} L^+\scrG_0/L^+\scrG_0$ (resp. )
\end{proof}

We can give a moduli description of the Schubert variety $S_i$. For $S_i\subset \Gr_{\scrG_I}$, it is the functor on the category of $k-$schemes whose $R$ valued points are the set of lattices:
\[
S_i(R)=\{\calL_I\in \Gr_{\scrG_I}\mid {\rm inv}(\lambda_I\otimes R, \calL_0)\in -{\rm Adm}_I(\mu_{n-i,i}) \}.
\]
We will now consider a resolution $D_i$ of $S_i$. For $S_i\subset \Gr_{\scrG_0}$, let $D_i$ be the functor that sends $R$ to the set of pairs of lattices $(\calL_0,\calL_0')$ such that
\[
\begin{array}{l}
(1).  \calL_0\in S_i(R);\\
(2).  \calL_0'\subset u^{-1} {\calL_0'}^{\vee}\subset u^{-1}\calL_0' ~\text{such that}~ u^{-1} {\calL_0'}^{\vee}/\calL_0' ~\text{is locally free of rank $n-2i$}. \\
(3). \lambda_0\subset \calL_0' ~\text{such that}~ \calL_0'/\lambda_0 ~\text{is locally free of rank $i$}.\\
(4). \calL_0\subset \calL_0' ~\text{such that}~ \calL_0'/\calL_0 ~\text{is locally free of rank $i$}.
\end{array}
\]
For $S_i\subset \Gr_{\scrG_m}$, a resolution $D_i$ of $S_i$ is the functor that sends $R$ to the set of pairs of lattices $(\calL_m,\calL_m')$ such that
\[
\begin{array}{l}
(1).  \calL_m\in S_i(R);\\
(2).  \calL_m'\subset u^{-1} {\calL_m'}^{\vee}\subset u^{-1}\calL_m' ~\text{such that}~ u^{-1} {\calL_0'}^{\vee}/\calL_0' ~\text{is locally free of rank $2i$}. \\
(3). \calL_m'\subset \lambda_m ~\text{such that}~ \lambda_m/\calL_m' ~\text{is locally free of rank $i$}.\\
(4). \calL_m'\subset \calL_m ~\text{such that}~ \calL_m/\calL_m' ~\text{is locally free of rank $i$}.
\end{array}
\]
Denote by $\tilde{S}_i=L^+\calG_I\times^{L^+Q_i} \overline{(L^+Q_i e^{\lambda_i}L^+\calG_I)}/L^+\calG_I$, where $Q_i$ is the parahoric subgroup that stabilizes the standard lattice chain $\lambda_0\subset \lambda_i$ for $I=\{0\}$ resp. $\lambda_{m-i}\subset \lambda_m$ for $I=\{m\}$. We have a natural morphism
\[
m: \tilde{S}_i\rightarrow S_i
\] 
given by the multiplication.
\begin{Proposition}
For $I=\{0\}, n=2m+1$ or $I=\{m\}, n=2m$, there is a natural morphism $\phi: \tilde{S}_i\rightarrow D_i$ such that the following diagram
\begin{center}
    \begin{tikzcd}
        \tilde{S}_i \arrow[rr, "\phi"] \arrow[dr, "m"'] & & D_i \arrow[dl, "{\rm pr}_1"] \\
        & S_i &
    \end{tikzcd}
\end{center}
is commutative and $\phi$ is an isomorphism.
 \end{Proposition} 
 
 \begin{proof}
 Mimic the proof of \cite[Proposition 4.4]{Richarz}. Set $\lambda_I'=\lambda_I\cap (e^{-\mu_i}\lambda_I)$ for $I=\{0\}$ or $I=\{m\}$. There is a natural morphism $\phi: \tilde{S}_i\rightarrow D_i$	given by 
 \[
 (g_1, g_2)\cdot (Q_i(R\lp t\rp), \calG_I(R\lp t\rp))\mapsto (g_1g_2\cdot \lambda_I\otimes_k R, g_1\cdot\lambda_I'\otimes_k R).
 \] It is easy to see that the diagram commutes. %To show $\phi$ is isomorphism, it is enough to show that $\tilde{S}_i$ and $D_i$ are irreducible and smooth of the same dimension by \cite{Richarz}. 
 \end{proof}
 
Now we want to connect the resolution $D_i$ to the special fiber of splitting models $\Mspl$. Consider the smooth irreducible component $Z\subset \Mspl\otimes k$, where $Z=\overline{X_{s,s}}=\{(\calF,\calG)\in \Mspl\otimes k\mid \{\calG,\calG\}=0\}$. Here $\{\ ,\ \}$ is the modified pairing in Definition \ref{G2} when $I=\{m\}, n=2m$, and in \cite[\S 2.2]{BH} when $I=\{0\}, n=2m+1$.

\begin{Proposition}\label{commdiag}
For $I=\{0\}, n=2m+1$ (resp. $I=\{m\}, n=2m$), there is an isomorphism $\varphi: Z\rightarrow D_s$ given by $(\calF, \calG)\mapsto (u^{-1}\calL_{\calF}, u^{-2}\calL_{\calG})$ (resp. $(\calF, \calG)\mapsto (u^{-1}\calL_{\calF}, u\calL_{\calG}^\vee)$), such that the following diagram is commutative:
\begin{center}
    \begin{tikzcd}
        Z \arrow[r, "\varphi"] \arrow[d, "{\rm pr_1}"'] & D_s \arrow[d, "{\rm pr_1}"] \\
        \Mloc\otimes k \arrow[r, "\simeq"] & S_s
    \end{tikzcd}
\end{center}		
\end{Proposition}

\begin{proof}
For $I=\{0\}, n=2m+1$, a point $(\calF,\calG)\in \Mspl\otimes k$ satisfies condition $(1)$-$(4)$ in \cite[\S 2.1]{Zac1}. Set $\calL_\calF, \calL_\calG$ the inverse image of $\calF, \calG$, condition (1)-(4) are equivalent to 
\[
\begin{array}{l}
(1).\quad u^2\lambda_0\subset \calL_\calG\subset \calL_\calF\subset \lambda_0,\\
(2).\quad \calL_\calF=u^2\calL_\calF^\vee,\\
(3).\quad u\calL_\calF\subset \calL_\calG,\\
(4).\quad u\calL_\calG\subset u^2\lambda_0.
\end{array}
\] 
Note that the modified pairing $\{\ ,\ \}=u^{-2}\phi: (u\lambda_0/u^2\lambda_0)\times (u\lambda_0/u^2\lambda_0)\rightarrow k$ gives $\calL_{\calG^{\bot'}}=u^3 \calL_{\calG}^\vee$. Therefore, by setting $(\calL_0,\calL_0')=(u^{-1}\calL_{\calF}, u^{-2}\calL_{\calG})$, a point $(\calF,\calG)\in Z$ is equivalent to $(\calL_0,\calL_0')$ satisfying
\[
\begin{array}{l}
(1).\quad u\lambda_0\subset u\calL_0'\subset \calL_0\subset u^{-1}\lambda_0,\\
(2).\quad \calL_0=\calL_0^\vee,\\
(3).\quad \calL_0\subset \calL_0',\\
(4).\quad \calL_0'\subset u^{-1}\lambda_0.\\
(5).\quad \calL_0'\subset u^{-1}{\calL_0'}^\vee
\end{array}
\] 
with rank($\calL_0'/\lambda_0)=s$, rank($\calL_0/u\lambda_0)=n$. We claim that $(\calL_0,\calL_0')\in D_s$. Indeed, it is easy to see that ${\calL_0'}^\vee\subset \lambda_0^\vee=\lambda_0\subset \calL_0'$. Since $\lambda_0\subset \calL_0'\subset u^{-1}{\calL_0'}^\vee$ and rank($u^{-1}{\calL_0'}^\vee/\lambda_0)=$rank($\lambda_0/u{\calL_0'})=n-s$, we obtain rank($u^{-1}{\calL_0'}^\vee/\calL_0')=n-2s$. Similarly, rank($\calL_0'/\calL_0)=$rank($u^{-1}\lambda_0/\calL_0)-$rank($u^{-1}\lambda_0/\calL_0')=s$.

Conversely, for any point $(\calL_0,\calL_0')\in D_s$, we get a pair of points $(\calF, \calG)$ given by $\calF=\calL_0/u\lambda_0, \calG=\calL_0'/\lambda_0$. To prove $(\calF, \calG)\in Z$, it is enough to show that $u\calL_0'\subset \calL_0, \calL_0\subset u^{-1}\lambda_0, \calL_0'\subset u^{-1}\lambda_0$ with correct ranks. Consider the dual of $\calL_0'\subset u^{-1}{\calL_0'}^\vee$. We get $u\calL_0'\subset {\calL_0'}^\vee\subset \calL_0^\vee=\calL_0$ and $u\calL_0'\subset {\calL_0'}^\vee\subset \lambda_0^\vee=\lambda_0$. Thus, we have $\calL_0'\subset u^{-1}\lambda_0$ and $\calL_0\subset u^{-1}\lambda_0$ by $\calL_0\subset \calL_0'$. It is easy to see that the diagram is commutative. 

The proof of the $\pi$-modular case is similar. We only want to mention that the spin condition is already shown in the condition $\calL_m\in S_i(R)$.
\end{proof}

\begin{Remark}{\rm 
Proposition \ref{commdiag} shows that the restriction $\tau\mid_Z: Z\rightarrow \Mloc\otimes k$ is a surjective birational projection morphism, and can be viewed as the affine Demazure resolution $D_s\rightarrow S_s$. By viewing the stratification $X_{h,l}$ and the Schubert cells $C_i$, we have
\[
Z=\coprod_{i=1}^s X_{i,s}\rightarrow \Mloc=\coprod_{i=1}^s C_i,\quad {\text for}~ I=\{0\}
\]
\[
Z=\coprod_{i=s \, \text{mod} \, 2} X_{i,s}\rightarrow \Mloc=\coprod_{i=s \, \text{mod} \, 2} C_i,\quad {\text for}~ I=\{m\} 
\]
such that the inverse image of $C_i$ is $X_{i,s}$.
}\end{Remark}

\Addresses
\end{document}